\documentclass[12pt]{article}

\usepackage{setspace}

\usepackage{dsfont} % Used for indicator functions \1_Q
\usepackage{amsmath,amssymb,amsthm}
\usepackage{epsfig}
\usepackage[utf8]{inputenc}
\usepackage[T1]{fontenc}
\usepackage{psfrag}
\usepackage{graphicx} % graphics, see http://mactex-wiki.tug.org/wiki/index.php/Graphics
\usepackage{caption}
\usepackage{epstopdf} % gerer les .eps
\usepackage{esint} % integral signs

\usepackage{tikz}
\usetikzlibrary{calc}

\begin{document}

\newtheorem{theorem}{Theorem}
\newtheorem{algorithm}{Algorithm}
\newtheorem{proposition}[theorem]{Proposition}
\newtheorem{remark}[theorem]{Remark}
\newtheorem{lemma}[theorem]{Lemma}
\newtheorem{corollary}[theorem]{Corollary}
\newtheorem{definition}[theorem]{Definition}

\newcommand{\RR}{\mathbb{R}}
\newcommand{\NN}{\mathbb{N}}
\newcommand{\ZZ}{\mathbb{Z}}
\newcommand{\PP}{\mathbb{P}}
\newcommand{\EE}{\mathbb{E}}
\newcommand{\Var}{\mathbb{V}{\rm ar}}
\newcommand{\Cov}{\mathbb{C}{\rm ov}}
\newcommand{\Id}{\text{Id}}
\newcommand{\dive}{\operatorname{div}}
\newcommand{\eps}{\varepsilon}
\newcommand{\1}{\mathds{1}}

\newcommand{\dps}{\displaystyle}
\newcommand{\dis}{\displaystyle}

\newcommand{\D}{{\cal{D}}}

\def\longrightharpoonup{\relbar\joinrel\rightharpoonup}

\newcommand{\fl}[1]{\textbf{Fred: #1}}
%\newcommand{\wm}[1]{\textbf{William : #1}}

%\doublespacing

\title{Special Quasirandom Structures: a selection approach for stochastic homogenization}
\author{C. Le Bris, F. Legoll and W. Minvielle\\
{\footnotesize \'Ecole des Ponts and INRIA,}\\
{\footnotesize 6 et 8 avenue Blaise Pascal, 77455 Marne-La-Vall\'ee Cedex 2, France}\\
{\footnotesize \tt \{lebris,william.minvielle\}@cermics.enpc.fr}\\
{\footnotesize \tt legoll@lami.enpc.fr}\\
%{\footnotesize $^2$ INRIA Rocquencourt, MATHERIALS research-team,}\\
%{\footnotesize Domaine de Voluceau, B.P. 105,
%78153 Le Chesnay Cedex, France}\\
%{\footnotesize $^3$ Laboratoire Navier, \'Ecole Nationale des Ponts et
%Chauss\'ees, Universit\'e Paris-Est,}\\
%{\footnotesize 6 et 8 avenue Blaise Pascal, 77455 Marne-La-Vall\'ee
%Cedex 2, France}\\
%{\footnotesize \tt legoll@lami.enpc.fr}}
}
\date{\today}

\maketitle

\abstract{We adapt and study a variance reduction approach for the homogenization of elliptic equations in divergence form. The approach, borrowed from atomistic simulations and solid-state science~\cite{vonPezoldDickFriakNeugebauer2010,WeiFerreiraBernardZunger1990,ZungerWeiFerreiraBernard1990}, consists in selecting random realizations that best satisfy some statistical properties (such as the volume fraction of each phase in a composite material) usually only obtained asymptotically.

We study the approach theoretically in some simplified settings (one-dimensional setting, perturbative setting in higher dimensions), and numerically demonstrate its efficiency in more general cases.}

%\tableofcontents

\section{Introduction}

\subsection{Overview}

In this article, we adapt, theoretically study and numerically test a specific variance reduction approach for the numerical homogenization of an elliptic equation with heterogeneous random coefficients.

The equation we consider is the following scalar elliptic equation in divergence form
\begin{equation}
\label{eq:pb0-stoc}
-\dive\left(A\left(\frac{\cdot}{\varepsilon},
\omega\right) \nabla u^\varepsilon(\cdot, \omega) \right) = f \ \ \text{in ${\mathcal D}$}, 
\qquad
u^\varepsilon(\cdot,\omega) = 0 \ \ \text{on $\partial {\mathcal D}$},
\end{equation}
set on a bounded regular domain $\mathcal D$ in $\RR^d$ (for some $d \geq 1$), with a deterministic function $f \in H^{-1}(\mathcal D)$ in the right-hand side. The field $A$ is a fixed matrix-valued random field. It is assumed to be uniformly elliptic, uniformly bounded and stationary in a discrete sense. All this is made precise in Section~\ref{sec:intro_2}. Since the parameter $\varepsilon$ in~\eqref{eq:pb0-stoc} is assumed small, the coefficient $\dps A\left(\frac{\cdot}{\varepsilon},\omega\right)$ is oscillatory and~\eqref{eq:pb0-stoc} is challenging to solve numerically. On the other hand, the problem is theoretically well understood, as is recalled below.

In the numerical practice, the traditional approach to approximate the solution $u^\varepsilon(\cdot,\omega)$ to~\eqref{eq:pb0-stoc} is to consider (for any $p \in \RR^d$), and solve, the so-called corrector problem
\begin{equation}
\label{eq:correcteur-complet}
\left\{
\begin{array}{l}
\dps -\dive\left[A(p + \nabla w_p)\right] = 0 \quad
\text{in } \RR^d \ \text{ almost surely},
\vspace{6pt}\\
\dps \int_Q {\mathbb E}(\nabla w_p) = 0, \quad \text{$\nabla w_p$ is stationary in the sense of~\eqref{eq:stationnarite-disc} below},
\end{array}\right.
\end{equation}
associated to~\eqref{eq:pb0-stoc}. The solution to~\eqref{eq:correcteur-complet} gives the deterministic and constant coefficient $A^\star$ of the homogenized equation that in turn serves for the approximation of~\eqref{eq:pb0-stoc}. We refer to Section~\ref{sec:intro_2} below for details.

Since~\eqref{eq:correcteur-complet} is a problem set on the entire space $\RR^d$, it is necessary to truncate it on a bounded domain, and to complement it with appropriate boundary conditions. In practice, it is standard to consider the problem
\begin{equation} 
\label{eq:correcteur-random-N}
-\dive \Big( A(\cdot,\omega) \left( p +  \nabla w_p^N(\cdot,\omega)\right) \Big) = 0,
\quad
w_p^N(\cdot,\omega) \ \mbox{is $Q_N$-periodic},
\end{equation}
where, say, $Q_N = (0,N)^d$. The deterministic homogenized matrix $A^\star$ is then approximated by the random variable $A^{\star}_N(\omega)$ defined by
\begin{equation}
\label{eq:AstarN}
\forall p \in \RR^d, \quad A^\star_N(\omega) \ p = \frac{1}{|Q_N|} \int_{Q_N} A(\cdot,\omega)(p + \nabla w^N_p(\cdot,\omega)).
\end{equation}
This approximate homogenized coefficient $A^\star_N(\omega)$ is then evaluated using the Monte-Carlo method. Random realizations of the environment, namely the matrix coefficient $A(y,\omega)$, are considered within the truncated domain $Q_N$. For each of these environments,~\eqref{eq:correcteur-random-N} is solved and the matrix $A^\star_N(\omega)$ is computed using~\eqref{eq:AstarN}. The homogenized coefficient $A^\star$ is eventually approximated as an empirical mean over several realizations of $A^\star_N(\omega)$. More details are given below in Section~\ref{sec:bp}.

\medskip

The purpose of this article is to reduce the variance of the approximation of $A^\star$.

\medskip

For this purpose, we borrow a variance reduction approach originally introduced in a completely different context, namely that of atomistic simulations for microscopic solid state science. In the series of articles~\cite{vonPezoldDickFriakNeugebauer2010,WeiFerreiraBernardZunger1990,ZungerWeiFerreiraBernard1990}, an approach is indeed described that selects some particular random realizations of the environment, based on some selection criteria derived from asymptotic properties. Intuitively, the approach aims at considering only realizations that, for $N$ fixed, {\em already} satisfy properties that are usually only obtained in the asymptotic limit $N \to \infty$. The approach carries the name SQS, abbreviation of {\em Special Quasirandom Structures}. Its principles share some similarity with those underlying another classical variance reduction technique, namely {\em stratified sampling}.

We aim at adapting this approach to our context, at studying it theoretically in some simple situations, and testing it numerically in more general situations.

For the sake of completeness, we mention that we have already studied the theoretical properties and the practical performance of several variance reduction methods for numerical random homogenization in some previous works of ours. The classical approach of {\em antithetic variables}, an approach that is quite generic and does not require nor exploit knowledge of the specific structure of the random problem at hand, has been considered in~\cite{BlancCostaouecLeBrisLegoll2012a,BlancCostaouecLeBrisLegoll2012b,CostaouecLeBrisLegoll2010,LegollMinvielle2015a}. The significantly more elaborate (and thus more efficient) approach of {\em control variates} is the subject of~\cite{LegollMinvielle2015b}. That approach requires a better knowledge of the problem considered, and is not always amenable to fully generic situations.

\medskip

Our article is articulated as follows. 

\medskip

In the remainder of this introductory section, we present the basics of the theoretical setting (in Section~\ref{sec:intro_2}) and of the numerical approximation method (in Section~\ref{sec:bp}) for the homogenization of the random equation~\eqref{eq:pb0-stoc}.

In Section~\ref{sec:SQS_building}, we introduce the variance reduction approach we consider. For pedagogic purposes, we first briefly expose the approach in the context of solid state physics it has originally been introduced in. This is the purpose of Section~\ref{sec:chemistry}. In Section~\ref{sec:derivation_SQS}, we formally derive the specifics of our variance reduction approach using a perturbative setting. This formal derivation provides the motivation for the general so-called SQS conditions that we use in the sequel of the work. Section~\ref{sec:eval} presents how we compute these conditions in practice. Section~\ref{sec:algo} contains the pseudo-code of our approach, along with some comments.

The theoretical analysis of the approach is the purpose of Section~\ref{sec:theorique}. We begin by proving, in a fairly general situation (in any ambient dimension), that the approximation provided by our approach (at least the simplest variant of our approach) converges to the homogenized coefficient $A^\star$ when the truncated domain converges to the whole space (see Theorem~\ref{thm:convergence} in Section~\ref{sec:consistency}). Next, in Section~\ref{sec:theorique_degrade}, we investigate more thoroughly the one-dimensional setting, where we can indeed completely analyze our approach and actually prove its efficiency.

Our final Section~\ref{sec:numerics} contains numerical tests. First, since it is often necessary to enforce the desired conditions up to some tolerance (see Remark~\ref{rem:definition} below), we investigate in Section~\ref{sec:selection} how this tolerance affects the quality of the approximation and the efficiency of the approach. We observe there that the approach is robust in this respect.

In Section~\ref{sec:convergence}, we illustrate on a prototypical situation the efficiency of our approach and scrutinize its sensitivity and the various sources of error involved. The conclusions are the following. The systematic error is kept approximately constant by the approach (it might even be reduced), while the variance is reduced by several orders of magnitude. The more conditions we impose on the microstructures, the smaller the variance. The total error is always reduced. Such an efficiency is achieved at almost no additional cost with respect to the classical Monte Carlo algorithm. 

In order to demonstrate the versatility of the approach, we apply it in Section~\ref{sec:mech} to a case with a way more general geometry of microstructures. There again, the approach provides a significant reduction of the variance. 

We conclude this overview by emphasizing that, although the approach introduced in this article is applied to the simple linear elliptic equation~\eqref{eq:pb0-stoc}, there is no reason to believe that it cannot be applied for a large class of partial differential equations with random coefficients. Indeed, the principles of the approach do not depend upon the specific form of the equation.

\subsection{Theoretical setting}
\label{sec:intro_2}

To begin with, we introduce the basic setting of stochastic homogenization. We refer to the seminal works~\cite{Kozlov1978,PapanicolaouVaradhan1981}, to~\cite{EngquistSouganidis2008} for a general, numerically oriented presentation and to~\cite{BensoussanLionsPapanicolaou1978,CioranescuDonato1999,JikovKozlovOleinik1994} for classical textbooks. We also refer to~\cite{LeBris2010} and the review article~\cite{AnantharamanCostaouecLeBrisLegollThomines2011} (and the extensive bibliography therein) for a presentation of our particular setting.

Throughout this article, $(\Omega, {\mathcal F}, \PP)$ is a probability space and we denote by $\dps \EE(X) = \int_\Omega X(\omega) d\PP(\omega)$ the expectation of any random variable $X\in L^1(\Omega, d\PP)$. We next fix $d\in {\mathbb N}^\star$ (the ambient physical dimension), and assume that the group $(\ZZ^d, +)$ acts on $\Omega$. We denote by $(\tau_k)_{k\in \ZZ^d}$ this action, and assume that it preserves the measure $\PP$, that is, for all $\dps k\in \ZZ^d$ and all $E \in {\cal F}$, $\dps \PP(\tau_k E) = \PP(E)$. We assume that the action $\tau$ is {\em ergodic}, that is, if $E \in {\mathcal F}$ is such that $\tau_k E = E$ for any $k \in \ZZ^d$, then $\PP(E) = 0$ or 1. In addition, we define the following notion of stationarity (see~\cite{LeBris2010}): a function $F \in L^1_{\rm loc}\left(\RR^d, L^1(\Omega)\right)$ is {\em stationary} if
\begin{equation}
\label{eq:stationnarite-disc}
\forall k \in \ZZ^d, \quad
F(x+k, \omega) = F(x,\tau_k\omega)
\quad
\text{a.e. in $x$ and a.s.}
\end{equation}
In this setting, the ergodic theorem~\cite{Shiryayev1984} can be stated as follows: 
%il est possible que les versions discretes soient dans~\cite{Krengel1985}~\cite{Tempelman1972}

{\it Let $F\in L^\infty\left(\RR^d, L^1(\Omega)\right)$ be a stationary random variable in the above sense. For $k = (k_1,k_2, \dots, k_d) \in \ZZ^d$, we set $\displaystyle |k|_\infty = \max_{1\leq i \leq d} |k_i|$. Then
$$
\frac{1}{(2N+1)^d} \sum_{|k|_\infty\leq N} F(x,\tau_k\omega)
\mathop{\longrightarrow}_{N\rightarrow \infty}
\EE\left(F(x,\cdot)\right) \ \ \mbox{in $L^\infty(\RR^d)$, almost surely}.
$$
This implies (denoting by $Q=(0,1)^d$ the unit cube in $\RR^d$) that
$$
F\left(\frac x \varepsilon ,\omega \right)
\mathop{\longrightharpoonup}_{\varepsilon \rightarrow 0}^\star
\EE\left(\int_Q F(x,\cdot)dx\right) \ \ \mbox{in $L^\infty(\RR^d)$, almost surely}.
$$
}

Besides technicalities, the purpose of the above setting is simply to formalize that, even though realizations may vary, the function $F$ at point $x \in \RR^d$ and the function $F$ at point $x+k$, $k \in \ZZ^d$, share the same law. In the homogenization context, this means that the local, microscopic environment (encoded in the matrix field $A$ in~\eqref{eq:pb0-stoc}) is everywhere the same \emph{on average}. From this, homogenized, macroscopic properties follow. 

\medskip

We consider problem~\eqref{eq:pb0-stoc}, which we recall here for convenience:
$$
-\dive\left(A\left(\frac{\cdot}{\varepsilon}, \omega\right) \nabla u^\varepsilon(\cdot, \omega) \right) = f \ \ \text{in ${\mathcal D}$}, 
\qquad
u^\varepsilon(\cdot,\omega) = 0 \ \ \text{on $\partial {\mathcal D}$}.
$$
The random matrix $A$ is assumed stationary in the sense of~\eqref{eq:stationnarite-disc}. We also assume that $A$ is bounded and coercive, that is, there exist two scalars $0<c \leq C < \infty$ such that, almost surely,
$$
%\begin{equation}
%\label{eq:elliptic}
\| A (\cdot,\omega) \|_{L^\infty(\RR^d)} \leq C
\quad \text{and} \quad
\forall \xi \in \RR^d, \quad \xi^T A(x,\omega) \xi \geq c \ \xi^T \xi
\quad \text{a.e.}
%\end{equation}
$$
In this specific setting, the solution $u^\eps(\cdot,\omega)$ to~\eqref{eq:pb0-stoc} almost surely converges (when $\eps$ goes to 0) to the solution $u^\star$ to the homogenized problem
\begin{equation}
\label{eq:pb0-stoc-homog}
-\dive\left(A^\star\nabla u^\star \right) = f \ \ \text{in ${\mathcal D}$}, 
\qquad
u^\star = 0 \ \ \text{on $\partial {\mathcal D}$}.
\end{equation}
The convergence of $u^\eps(\cdot,\omega)$ to $u^\star$ holds weakly in $H^1({\cal D})$ and strongly in $L^2({\cal D})$.

The homogenized matrix $A^\star$ in~\eqref{eq:pb0-stoc-homog} is deterministic, and given by an expectation of an integral involving the so-called corrector function, that solves a random auxiliary problem set on the {\em entire} space. It is given by
\begin{equation}
\label{eq:matrice-homogeneisee}
\forall p \in \RR^d, \quad
A^\star \, p = {\mathbb E} \left[ \int_Q A(x,\cdot) \ (p + \nabla w_p(x,\cdot)) \, dx \right],
\end{equation}
where we recall that $Q=(0,1)^d$ and where, for any vector $p \in \RR^d$, the \emph{corrector} $w_p$ is the unique solution (up to the addition of a random constant) in $L^2(\Omega; L^2_{\rm loc}(\RR^d))$ with gradient in $L^2(\Omega; L^2_{\rm unif}(\RR^d))^d$ of the corrector problem~\eqref{eq:correcteur-complet}. We have used the notation $L^2_{\rm unif}(\RR^d)$ for the \emph{uniform} $L^2$ space, that is the space of functions for which, say, the $L^2$ norm on a ball of unit size is bounded from above independently of the center of the ball. 

\subsection{Numerical approximation of the homogenized matrix}
\label{sec:bp}

As briefly mentioned above, the corrector problem~\eqref{eq:correcteur-complet} is set on the {\em entire} space $\RR^d$, and is therefore challenging to solve. Approximations are in order. In practice, the deterministic matrix $A^\star$ is approximated by the random matrix $A^\star_N (\omega)$ defined by~\eqref{eq:AstarN}, which is obtained by solving the corrector problem~\eqref{eq:correcteur-random-N} on a \emph{truncated} domain, say the cube $Q_N = (0,N)^d$. Although $A^\star$ itself is a deterministic object, its practical approximation $A^\star_N$ is random. It is only in the limit of infinitely large domains $Q_N$ that the deterministic value is attained. As shown in~\cite{BourgeatPiatnitski2004}, we indeed have 
\begin{equation}
\label{eq:BP_limit}
\lim_{N \to \infty} A^\star_N(\omega) = A^\star \quad \text{almost surely.} 
\end{equation}
As usual, the error $A^\star - A^\star_N(\omega)$ may be expanded as
\begin{equation}
\label{eq:error-decomposition}
A^\star - A_N^\star(\omega)
=
\Big(
A^\star - \EE \left[ A_N^\star \right]
\Big)
+
\Big(
\EE \left[ A_N^\star \right]
- A_N^\star(\omega)
\Big),
\end{equation}
that is the sum of a {\em systematic} error and of a {\em statistical} error (the first and second terms in the above right-hand side, respectively). 

A standard technique to compute an approximation of $\EE \left[A^\star_N \right]$ is to consider $M$ independent and identically distributed realizations of the field $A$, solve for each of them the corrector problem~\eqref{eq:correcteur-random-N} (thereby obtaining i.i.d. realizations $A^{\star,m}_N(\omega)$, for $1 \leq m \leq M$) and compute the Monte Carlo approximation
\begin{equation}
\label{eq:MC_estim_homog}
\EE \left[ \left( A^\star_N \right)_{ij}\right]
\approx
I^{\rm MC}_M(\omega) := \frac{1}{M} \sum_{m=1}^M \left( A^{\star,m}_N(\omega) \right)_{ij}
\end{equation}
for any $1 \leq i,j \leq d$. In view of the Central Limit Theorem, we know that $\EE \left[ \left( A^\star_N \right)_{ij} \right]$ asymptotically lies within the confidence interval
$$
%\begin{equation}
%\label{eq:confidence_interval}
\left[
I^{\rm MC}_M - 1.96 \frac{\sqrt{\Var \left[ \left( A^\star_N \right)_{ij} \right]}}{\sqrt{M}}
,
I^{\rm MC}_M + 1.96 \frac{\sqrt{\Var \left[ \left( A^\star_N \right)_{ij} \right]}}{\sqrt{M}}
\right]
%\end{equation}
$$
with a probability equal to 95 \%. 

For simplicity, and because this is overwhelmingly the case in the numerical practice, we have considered in~\eqref{eq:correcteur-random-N} {\em periodic} boundary conditions. These are the conditions we adopt throughout our study. It is to be remarked, however, that other boundary conditions may be employed. Likewise, other slightly modified forms of equation~\eqref{eq:correcteur-random-N} may be considered. The specific choice of approximation technique is motivated by considerations about the decrease of the systematic error in~\eqref{eq:error-decomposition}. Several recent mathematical studies have clarified this issue. In addition, in the particular case of periodic boundary conditions~\eqref{eq:correcteur-random-N}, it has been recently established in~\cite[Theorem 2]{GloriaNeukammOtto2015} that the statistical error in~\eqref{eq:error-decomposition} decays like $N^{-d/2}$ while the systematic error in~\eqref{eq:error-decomposition} scales as $N^{-d} (\log N)^d$. Both estimates have been established for the {\em discrete variant} of the problem. A similar decay of the statistical error has also been established for the continuous case we consider in the present article (see~\cite[Theorem 1]{GloriaOtto2015a} and~\cite[Theorem 1.3 and Proposition 1.4]{Nolen2014}).

\section{Variance reduction approach}
\label{sec:SQS_building}

\subsection{Original formulation of the SQS approach}
\label{sec:chemistry}

The variance reduction approach we elaborate upon in this article has been originally introduced for a slightly different purpose in atomistic solid-state science~\cite{vonPezoldDickFriakNeugebauer2010,WeiFerreiraBernardZunger1990,ZungerWeiFerreiraBernard1990}. 

In order to convey to the reader the intuition of the original approach, we consider here a simple one-dimensional setting, which nevertheless illustrates the difficulties of a generic problem. We consider a linear chain of atomistic sites of two species $A$ and $B$ which interact by the interaction potentials $V_{AA}$, $V_{AB}$ and $V_{BB}$ with obvious notation. For simplicity we consider only nearest neighbour interaction. The atomic sites are occupied by a single species randomly chosen between $A$ and $B$. A typical random configuration of the ``material'' therefore reads as an infinite sequence of the type $\cdots ABBAAABBAAAA \cdots$

In order to compute the energy per unit particle of that atomistic system, one has to consider all possible such infinite sequences, and for each of them its normalized energy
\begin{equation}
\label{eq:eq_V}
\lim_{N \to \infty} \frac1{2N +1} \sum_{i = - N}^N V_{X_{i+1}X_i},
\end{equation}
where $X_i$ denotes the species present at the $i$-th site for that particular configuration ($X_i \equiv A$ or $B$). The ``energy'' of the system is then defined as the {\em expectation} of~\eqref{eq:eq_V} over all possible configurations. Other quantities than~\eqref{eq:eq_V} may be considered, or may be simultaneously considered. 

In practice, one considers a presumably extremely large, finite $N$, truncates the infinite sequence over the finite length $2N+1$, and compute $$\displaystyle \frac1{2N +1} \sum_{i = - N}^N V_{X_{i+1}X_i}$$ for many (say $M$, where $M$ is also presumably large) configurations.

The approach introduced in~\cite{vonPezoldDickFriakNeugebauer2010,WeiFerreiraBernardZunger1990,ZungerWeiFerreiraBernard1990} consists in {\em selecting} specific configurations $(X_i)_{-N \leq i \leq N}$ of atomic sites that satisfy statistical properties usually obtained only in the limit of infinitely large $N$.

The first such statistical property is the volume fraction, namely the proportion of species $(A, B)$ present on average. If the sites are all occupied randomly with probability $1/2$ of $A$ and $1/2$ of $B$ (and assuming that all these random variables are independent), then obviously the volume fraction of $A$ is $1/2$ and so is that of $B$. Then, one only consider truncated sequences $(X_i)_{-N \leq i \leq N}$ that {\em exactly} reproduce that volume fraction. 

Similarly, again for such an evenly distributed proportion of $A$ and $B$, the energy of the entire infinite system evidently reads as $${\cal E} = \frac14\left[ V_{AA} + 2 V_{AB} + V_{BB} \right]$$ (recall that we only consider nearest-neighbour interactions). Thus, one only considers truncated sequences $(X_i)_{-N \leq i \leq N}$ which, in addition to exhibiting the exact volume fraction, have an average energy $\displaystyle \frac1{2N+1}\sum_{i = -N}^N V_{X_{i+1}X_i}$ which is {\em equal} to ${\cal E}$. And so on and so forth for other quantities of interest.

\medskip

Mathematically, this {\em selection} of suitable configurations among all the possible configurations classically considered in a Monte-Carlo sample amounts to replacing the computation of an expectation by that of a {\em conditional expectation}.

The simplistic model we have just considered for pedagogic purposes can of course be replaced by more elaborate models, with more sophisticated quantities to compute, and more demanding statistical quantities to condition the computations with. The bottom line of the approach remains the same, and we adapt it to design a variance reduction approach for numerical random homogenization.

In the next section, we derive the appropriate conditions, which we call the SQS conditions, for our specific context.

\subsection{Formal derivation of the SQS conditions using a perturbative setting}
\label{sec:derivation_SQS}

The purpose of this Section is to formally derive the SQS conditions that we use in the sequel. Such conditions can be easily intuitively understood. We however believe it is interesting to (formally) {\em derive} them in a particular case. The case we proceed with is a perturbative setting (although, we emphasize it, the conditions will be employed in the full general, not necessarily perturbative, setting).

\medskip

We assume throughout this section that the matrix valued coefficient $A$ in~\eqref{eq:pb0-stoc} reads as
\begin{equation}
\label{eq:A_form}
A_\eta(x, \omega) = C_0(x, \omega) + \eta \ \chi(x, \omega) C_1(x, \omega)
\end{equation}
for some presumably small scalar coefficient $\eta$, where
\begin{itemize}
\item $C_0$ and $C_1$ are two stationary, uniformly bounded matrix fields,
\item $C_0(\cdot, \omega) - C_1(\cdot, \omega)$ and $C_0(\cdot, \omega) + C_1(\cdot, \omega)$ are almost surely coercive,
\item $\chi$ is a stationary scalar field with values in $[-1,1]$.
\end{itemize}
Under these assumptions, for any $\eta \in (-1,1)$, the matrix $A_\eta$ is stationary, bounded and coercive. Intuitively, when $\eta$ is small, $A_\eta$ is a perturbation of the matrix-valued field $C_0(x, \omega)$.

\begin{remark}
The expression~\eqref{eq:A_form} models e.g. a two-phase composite material, where the phases are modelled by the coefficients $C_0$ and $C_1$, while $\chi$ is the indicator function of the first phase. 
\end{remark}

Let $p \in \RR^d$. The corrector problem~\eqref{eq:correcteur-complet} reads, in this particular setting, as
\begin{equation}
\label{eq:correcteur}
\left\{
\begin{array}{c}
-\dive \big[ (C_0 + \eta \chi C_1) (p+\nabla w_\eta)\big] = 0
\quad \text{in $\RR^d$},
\\
\dps \EE \int_Q \nabla w_\eta = 0, \quad \text{$\nabla w_\eta$ is stationary in the sense of~\eqref{eq:stationnarite-disc},}
\end{array}
\right.
\end{equation}
and the homogenized matrix~\eqref{eq:matrice-homogeneisee} is given by
\begin{equation}
\label{eq:def_Astar}
\forall p \in \RR^d, \quad
A^\star_\eta \, p = \EE \int_Q A_\eta(p + \nabla w_\eta).
\end{equation}
Note that, for the sake of clarity, we omit to write the dependency of $w_\eta$ with respect to $p$.

The truncated version of~\eqref{eq:correcteur} on the domain $Q_N$ is
\begin{equation}
\label{eq:correcteur_N}
\left\{
\begin{array}{c}
-\dive \big[ (C_0 + \eta \chi C_1) (p+\nabla w^N_\eta)\big] = 0
\quad \text{in $Q_N$},
\\ \noalign{\vskip 3 pt}
\text{$w^N_\eta(\cdot, \omega)$ is $Q_N$-periodic},
\end{array}
\right.
\end{equation}
and we approach the homogenized matrix~\eqref{eq:def_Astar} by
\begin{equation}
\label{eq:def_AstarN}
\forall p \in \RR^d, \quad
A^{\star,N}_\eta(\omega) \, p = \frac{1}{\left| Q_N \right|} \int_{Q_N} A_\eta(\cdot,\omega) \, \big( p + \nabla w^N_\eta(\cdot,\omega) \big).
\end{equation}

\subsubsection{Expansion in powers of $\eta$} 
\label{sec:dev_eta}

As $\eta$ goes to 0, we may now expand $A^{\star,N}_\eta(\omega)$ and $A^\star_\eta$ in powers of $\eta$. This expansion is classical (see for instance~\cite{BlancCostaouecLeBrisLegoll2012b,Costaouec2012b}). We only provide it here for the sake of consistency. The corrector expands as
\begin{equation}
\label{eq:toto}
\nabla w_\eta = \nabla w_0 + \eta \nabla u_1 + \eta^2 \nabla u_2 + o(\eta^2).
\end{equation}
This expansion holds in $L^2(\Omega; L^2_{\rm unif}(\RR^d))$. The functions $w_0$, $u_1$ and $u_2$ appearing in the expansion are respectively defined by the following systems of equations:
\begin{equation}
\label{eq:correcteur_cascade0}
\left\{
\begin{array}{l}
- \dive \left[ C_0 (p + \nabla w_0) \right] = 0 \quad \text{in $\RR^d$},
\\ \noalign{\vskip 3 pt}
\dps \EE \int_Q \nabla w_0 = 0, \quad \text{$\nabla w_0$ is stationary},
\end{array}
\right.
\end{equation}
\begin{equation}
\label{eq:correcteur_cascade1}
\left\{
\begin{array}{l}
- \dive \left[ C_0 \nabla u_1 \right] = \dive \left[ \chi C_1 (p+\nabla w_0) \right] \quad \text{in $\RR^d$},
\\ \noalign{\vskip 3 pt}
\dps \EE \int_Q \nabla u_1 = 0, \quad \text{$\nabla u_1$ is stationary},
\end{array}
\right.
\end{equation}
and
$$
%\begin{equation}
%\label{eq:correcteur_cascade2}
\left\{
\begin{array}{l}
- \dive \left[ C_0 \nabla u_2 \right] = \dive \left[ \chi C_1 \nabla u_1 \right] \quad \text{in $\RR^d$},
\\ \noalign{\vskip 3 pt}
\dps \EE \int_Q \nabla u_2 = 0, \quad \text{$\nabla u_2$ is stationary}.
\end{array}
\right.
%\end{equation}
$$
Inserting the expansion~\eqref{eq:A_form} of $A_\eta$ and~\eqref{eq:toto} of $w_\eta$ in~\eqref{eq:def_Astar}, we obtain
\begin{equation}
\label{eq:monde_reve}
A^\star_\eta = A^\star_0 + \eta A^\star_1 + \eta^2 A^\star_2 + o(\eta^2),
\end{equation}
with, for any $p \in \RR^d$,
\begin{eqnarray}
%\label{eq:A_star0}
\nonumber
A^\star_0 \, p &=& \EE \left[ \int_Q C_0 (p + \nabla w_0) \right],
\\
\label{eq:A_star1}
A^\star_1 \, p &=& \EE\left[ \int_Q \chi C_1 (p+ \nabla w_0) \right] + \EE \left[ \int_Q C_0 \nabla u_1 \right],
\\
%\label{eq:A_star2}
\nonumber
A^\star_2 \, p &=& \EE \left[ \int_Q \chi C_1 \nabla u_1 \right] + \EE \left[ \int_Q C_0 \nabla u_2 \right].
\end{eqnarray}
Likewise, we expand $w^N_\eta$ as
$$
\nabla w^N_\eta = \nabla w^N_0 + \eta \nabla u^N_1 + \eta^2 \nabla u^N_2 + o(\eta^2),
$$ 
with
\begin{equation}
\label{eq:cascade_N0}
\left\{
\begin{array}{l}
- \dive \left[ C_0 (p + \nabla w^N_0) \right]= 0 \quad \text{in $Q_N$},
\\ \noalign{\vskip 3 pt}
\text{$w^N_0(\cdot, \omega)$ is $Q_N$-periodic},
\end{array}
\right.
\end{equation}
\begin{equation}
\label{eq:cascade_N1}
\left\{
\begin{array}{l}
- \dive  \left[ C_0 \nabla u^N_1 \right] = \dive \left[ \chi C_1 (p+\nabla w^N_0) \right] \quad \text{in $Q_N$},
\\ \noalign{\vskip 3 pt}
\text{$u^N_1(\cdot, \omega)$ is $Q_N$-periodic},
\end{array}
\right.
\end{equation}
and
$$
%\begin{equation}
%\label{eq:cascade_N2}
\left\{
\begin{array}{l}
- \dive \left[ C_0 \nabla u^N_2 \right] = \dive \left[ \chi C_1 \nabla u^N_1 \right] \quad \text{in $Q_N$},
\\ \noalign{\vskip 3 pt}
\text{$u^N_2(\cdot, \omega)$ is $Q_N$-periodic}.
\end{array}
\right.
%\end{equation}
$$
The homogenized matrix $A^{\star,N}_\eta(\omega)$ therefore satisfies
\begin{equation}
\label{eq:monde_reel}
\Big| A^{\star,N}_\eta(\omega) - \Big[ A^{\star,N}_0(\omega) + \eta A^{\star,N}_1(\omega) + \eta^2 A^{\star,N}_2(\omega) \Big] \Big| \leq C \eta^3,
\end{equation}
where $C$ is independent of $\eta$, $N$ and $\omega$, and where the matrices $A^{\star,N}_0(\omega)$, $A^{\star,N}_1(\omega)$ and $A^{\star,N}_2(\omega)$ are defined by
\begin{eqnarray}
%\label{eq:A_star0_N}
\nonumber
A^{\star,N}_0(\omega) \, p &=& \frac{1}{|Q_N|} \int_{Q_N} C_0 (p + \nabla w^N_0),
\\ \noalign{\vskip 3 pt}
\label{eq:A_star1_N}
A^{\star,N}_1(\omega) \, p &=& \frac{1}{|Q_N|} \int_{Q_N} \chi C_1 (p+ \nabla w^N_0) + \frac{1}{|Q_N|} \int_{Q_N} C_0 \nabla u^N_1,
\\ \noalign{\vskip 3 pt}
%\label{eq:A_star2_N}
\nonumber
A^{\star,N}_2(\omega) \, p &=& \frac{1}{|Q_N|} \int_{Q_N} \chi C_1 \nabla u^N_1 + \frac{1}{|Q_N|} \int_{Q_N} C_0 \nabla u^N_2.
\end{eqnarray}

\subsubsection{SQS conditions}

In line with the motivation we have mentioned above in Section~\ref{sec:bp}, we are now in position to introduce the conditions that we use to {\em select} particular configurations of the environment within $Q_N$ for which we compute the solution to~\eqref{eq:correcteur_N}, and, in turn, compute the approximation~\eqref{eq:def_AstarN} of $A^\star_\eta$. Our conditions are based upon the comparison of~\eqref{eq:A_star1} and~\eqref{eq:A_star1_N}. 

\medskip

\begin{definition}
\label{def:SQS}
For finite fixed $N$, we say that an environment $\omega \in \Omega$ satisfies the SQS condition of
\begin{itemize}
\item order 0 if $A^{\star,N}_0(\omega) = A^\star_0$, that is to say, for any $p \in \RR^d$,
\begin{equation}
\label{eq:cond0}
\frac{1}{|Q_N|}\int_{Q_N} C_0(\cdot, \omega) (p + \nabla w^N_0(\cdot, \omega)) = \EE\left[\int_Q C_0 (p + \nabla w_0)\right],
\end{equation}
\item order 1 if $A^{\star,N}_1(\omega) = A^\star_1$, that is to say, for any $p \in \RR^d$,
\begin{multline}
\label{eq:cond1}
\frac{1}{|Q_N|}\int_{Q_N} \Big[ \chi(\cdot, \omega)C_1(\cdot, \omega) (p+ \nabla w^N_0(\cdot, \omega)) + C_0(\cdot, \omega) \nabla u^N_1(\cdot, \omega) \Big] \\ 
= \EE\left[ \int_Q\chi C_1 (p+ \nabla w_0) + C_0 \nabla u_1\right],
\end{multline}
\item order 2 if $A^{\star,N}_2(\omega) = A^\star_2$, that is to say, for any $p \in \RR^d$,
\begin{multline}
\label{eq:cond2}
\frac{1}{|Q_N|}\int_{Q_N} \Big[ \chi(\cdot, \omega)C_1(\cdot, \omega) \nabla u^N_1(\cdot, \omega) +C_0(\cdot, \omega) \nabla u^N_2 (\cdot, \omega) \Big] \\ 
= \EE\left[\int_Q\chi C_1 \nabla u_1 + C_0 \nabla u_2\right].
\end{multline}
\end{itemize}
\end{definition}

\begin{remark}
\label{rem:definition}
In full generality, we {\em do not} claim that there exist environments that satisfy these conditions. This might be the case that no such environment exists. One may for instance simply remark that a random variable that takes value $-1$ and $+1$ both with probability $1/2$ never has value zero, which is its expectation! In some situations, we therefore have to {\em relax} the above conditions (see Section~\ref{sec:algo} below), but we temporarily leave these technicalities aside and {\em assume} that suitable environments exist.
\end{remark}

Consider now the two expansions~\eqref{eq:monde_reve} and~\eqref{eq:monde_reel}. It is immediate to see, by subtraction, that 
$$
A^{\star, N}_\eta(\omega) - A^\star_\eta = (A^{\star,N}_0(\omega) - A^\star_0) + \eta(A^{\star,N}_1(\omega) - A^\star_1) + \eta^2(A^{\star,N}_2(\omega) - A^\star_2) + o(\eta^2).
$$ 
Therefore it is readily seen that, if the configuration $\omega$ satisfies the SQS conditions of Definition~\ref{def:SQS} up to the order $k$ included ($k = 0$, 1, 2 in our definition, but clearly one could consider higher order conditions derived likewise), then 
\begin{equation}
\label{eq:ordre_k}
A^{\star,N}_\eta(\omega)-A^\star_\eta = o(\eta^k),
\end{equation}
where the constant in the right-hand side is independent of $\eta$, $N$ and $\omega$. Taking the expectation over such configurations therefore formally provides a more accurate approximation of $A^\star_\eta$.

\medskip

Now that we have derived the conditions~\eqref{eq:cond0}--\eqref{eq:cond1}--\eqref{eq:cond2} (which we henceforth call the {\em SQS conditions}) in the perturbative setting, we will actually use them in the non perturbative setting, namely for a similar two-phase composite material, but with $\eta$ {\em not} small. Of course, a property such as~\eqref{eq:ordre_k} cannot be expected any longer since the homogenized matrix $A^\star$ is no longer a series in a small coefficient that encodes a perturbation. Nevertheless, it can be expected that selecting the configurations using these conditions may improve the approximation, in particular by reducing the variance. We show in Sections~\ref{sec:theorique} and~\ref{sec:numerics} that it is indeed the case, theoretically and experimentally.

For the time being, we need to make a {\em practical} observation. The right-hand side of conditions~\eqref{eq:cond0}--\eqref{eq:cond1}--\eqref{eq:cond2} need to be evaluated in order to practically encode the SQS conditions. In principle, those right-hand sides are exact expectations, that can only be determined using an asymptotic limit, and are therefore as challenging to compute in practice as $A^\star$ itself. 

We therefore need to restrict the generality of our setting~\eqref{eq:A_form} and consider cases where those right-hand sides are indeed amenable to a simple, inexpensive computation. This is the purpose of the next section.

\subsection{Practical evaluation of the SQS conditions}
\label{sec:eval}

In order to make our approach practical, we need, as mentioned above, to consider settings where the expectations present in the right-hand sides of~\eqref{eq:cond0}--\eqref{eq:cond1}--\eqref{eq:cond2} may be computed effectively.

\subsubsection{Condition of order 0}

We first consider~\eqref{eq:cond0} and its right-hand side
\begin{equation}
\label{eq:cond0_rhs}
\EE \left[ \int_Q C_0(x, \cdot) (p + \nabla w_0(x, \cdot)) dx \right].
\end{equation}
A natural assumption, which already covers a large portion of practically relevant situations, is
\begin{equation}
\label{eq:C_0_per}
C_0(x, \omega) = C_0(x) \quad \text{is a deterministic, $\ZZ^d$-periodic matrix.}
\end{equation}
The computation of~\eqref{eq:cond0_rhs} is then inexpensive since the solution $w_0$ to~\eqref{eq:correcteur_cascade0} is in fact the deterministic solution to
$$
- \dive \left[ C_0 (p + \nabla w_0) \right] = 0 \ \ \text{in $\RR^d$}, \quad \text{$w_0$ is $\ZZ^d$-periodic,}
$$
which is unique up to the addition of a constant.

In addition, when $N$ is an integer (and when the approximation chosen for~\eqref{eq:correcteur-complet} is the {\em periodic} approximation~\eqref{eq:correcteur-random-N}, as is indeed the case throughout this work), the solution to~\eqref{eq:cascade_N0} is $w_0^N \equiv w_0$ (up to an additive constant), and hence the condition~\eqref{eq:cond0} is systematically satisfied.

We henceforth assume that~\eqref{eq:C_0_per} holds, that $N$ is an integer, and that we proceed with the periodic approximation~\eqref{eq:correcteur-random-N}.

\subsubsection{Condition of order 1}

We next consider the SQS condition~\eqref{eq:cond1}. One possible assumption to make that condition practical is
\begin{equation}
\label{eq:C_0_0}
C_0(x, \omega) = C_0 \quad \text{is a deterministic, {\em constant} matrix.}
\end{equation}
Since $\nabla w_0 = 0$, the right-hand side of~\eqref{eq:cond1} reads
$$
%\begin{equation}
%\label{eq:cond1_temp}
\EE \left[ \int_Q \chi C_1 (p + \nabla w_0) + C_0 \nabla u_1 \right] = \int_Q \EE\left[ \chi C_1\right]p  + C_0 \EE \int_Q \nabla u_1,
%\end{equation}
$$
where the rightmost term vanishes in view of~\eqref{eq:correcteur_cascade1} and where the first term of the right-hand side may be computed using only characteristic properties of the environment considered. The condition~\eqref{eq:cond1} thus reads
\begin{equation}
\label{eq:cond1_bis}
\frac{1}{|Q_N|}\int_{Q_N} \chi(\cdot, \omega) C_1(\cdot, \omega)
= 
\EE\left[ \int_Q \chi C_1 \right].
\end{equation}

For instance, in a two-phase composite material mixing two {\em constant} and {\em deterministic} matrices $C_0$ and $C_1$, we have
$$
\EE \left[ \int_Q \chi C_1 \right] = \EE \left[ \int_Q \chi \right]C_1.
$$
This quantity obviously only depends on the {\em volume fraction} of the two phases (recall~\eqref{eq:A_form}). Proceeding likewise with the left-hand side of the condition~\eqref{eq:cond1}, we see that this condition reads
$$
%\begin{equation}
%\label{eq:cond1_simplified}
\frac{1}{|Q_N|}\int_{Q_N} \chi(x, \omega)dx = \EE\left[ \int_Q\chi\right].
%\end{equation}
$$
Interestingly (and not unexpectedly), we notice here that this condition on the volume fraction agrees with the condition we used to consider in the simple atomistic system of Section~\ref{sec:chemistry}.
 
\subsubsection{Condition of order 2}
 
We next proceed with condition~\eqref{eq:cond2}. In addition to~\eqref{eq:C_0_0}, we assume that
\begin{equation}
\label{eq:C_1_per}
C_1(x, \omega) = C_1(x) \quad \text{is a deterministic, $\ZZ^d$-periodic matrix,}
\end{equation}
and that
\begin{equation}
\label{eq:chi_form}
\chi(y, \omega) = \sum_{k \in \ZZ^d} X_k(\omega) \1_{Q+k}(y),
\end{equation}
where $X_k$ are identically distributed scalar random variables taking their values in $[-1,1]$. We also assume that
\begin{equation}
\label{eq:hyp_cov}
\mathcal{C} = \sum_{k \in \ZZ^d} \left| \Cov (X_0,X_k) \right| < \infty,
\end{equation}
which is obviously satisfied if $X_k$ are independent one from each other. 

\medskip

We then have the following result, which will be useful to make condition~\eqref{eq:cond2} practical. Its proof is postponed until Appendix~\ref{app:proof}.
\begin{lemma}
\label{lem:decompo}
Under the assumptions~\eqref{eq:C_0_0}, \eqref{eq:C_1_per}, \eqref{eq:chi_form} and~\eqref{eq:hyp_cov}, the solution $u_1$ to~\eqref{eq:correcteur_cascade1} satisfies
\begin{equation}
\label{eq:def_u1}
\nabla u_1(y, \omega) = \EE[X_0] \nabla \overline{u_1}(y) + \sum_{k \in \ZZ^d} \Big( X_k(\omega) - \EE[X_k] \Big) \nabla \phi_1(y-k),
\end{equation}
where $\phi_1$ is the (unique up to the addition of a constant) solution in $\{ v \in L^2_{\rm loc}(\RR^d), \ \ \nabla v \in (L^2(\RR^d))^d \}$ to
\begin{equation}
\label{eq:def_u1_k}
- \dive \left[ C_0 \nabla \phi_1 \right] = \dive \left[ \1_Q C_1p\right] \quad \text{in } \RR^d
\end{equation}
and $\overline{u_1}$ is the (unique up to the addition of a constant) solution to
\begin{equation}
\label{eq:def_u1_bar}
- \dive \left[ C_0 \nabla \overline{u_1} \right] = \dive \Big[ C_1 p \Big] \ \ \text{in $\RR^d$},
\quad 
\text{$\overline{u_1}$ is $\ZZ^d$-periodic}.
\end{equation}
The sum in~\eqref{eq:def_u1} is a convergent series in $L^2(Q \times \Omega)$. 
\end{lemma}

Using simpler arguments, we see that the solution $u_1^N$ to~\eqref{eq:cascade_N1} satisfies
\begin{equation}
\label{eq:def_u1^N}
\nabla u_1^N(y, \omega) = \EE[X_0] \nabla \overline{u_1}(y) + \sum_{k \in \ZZ^d \cap Q_N} \Big( X_k(\omega) - \EE[X_k] \Big) \nabla \phi_1^N(y-k),
\end{equation}
where $\overline{u_1}$ is defined by~\eqref{eq:def_u1_bar} and $\phi_1^N$ is the (unique up to the addition of a constant) solution to
\begin{equation}
\label{eq:def_u1^N_k}
- \dive \left[ C_0 \nabla \phi_1^N \right] = \dive \left[ \1_{Q}  C_1p \right] \ \ \text{in $Q_N$}, \quad \text{$\phi_1^N$ is $Q_N$-periodic}.
\end{equation}

\medskip

In practice, we can easily obtain an accurate approximation of $\phi_1$ since the right-hand side of~\eqref{eq:def_u1_k} has compact support. Truncating~\eqref{eq:def_u1_k} over a sufficiently large bounded domain (with homogeneous Dirichlet boundary conditions) provides such an accurate approximation. Given~\eqref{eq:C_0_0}, the right-hand side of Condition~\eqref{eq:cond2} rewrites $\dps \EE \left[ \int_Q \chi C_1 \nabla u_1 \right]$ since $\dps \EE \left[ \int_Q \nabla u_2 \right] = 0$. In view of~\eqref{eq:def_u1}, this quantity is in turn expanded as
\begin{eqnarray}
\nonumber
&&
\EE\left[\int_Q\chi C_1 \nabla u_1\right] 
\\
\nonumber
&=& 
(\EE[X_0])^2 \int_Q C_1 \nabla \overline{u_1} 
+
\sum_{k \in \ZZ^d} \EE\left[\int_Q X_0 \Big( X_k - \EE[X_k] \Big) C_1 \nabla \phi_1(\cdot-k) \right]
\\
\nonumber
&=& 
(\EE[X_0])^2 \int_Q C_1 \nabla \overline{u_1} 
\\
&&+
\sum_{k \in \ZZ^d} \EE\left[\int_Q \Big( X_0 - \EE[X_0] \Big) \Big( X_k - \EE[X_k] \Big) C_1 \nabla \phi_1(\cdot-k) \right],
\label{eq:cond2_revised_rhs}
\end{eqnarray}
where, as mentioned above, $\nabla \phi_1$ can be easily and accurately computed, while the series in $k \in \ZZ^d$ may be truncated in an efficient manner because of the rapid decay at infinity of $\nabla \phi_1$ (see~\cite[Lemma 3.1]{BlancCostaouecLeBrisLegoll2012b}).

We correspondingly expand the left-hand side of~\eqref{eq:cond2}. The second term vanishes, while the first term reads, in view of~\eqref{eq:def_u1^N},
\begin{eqnarray}
&&
\dps \frac{1}{|Q_N|} \int_{Q_N} \chi(y,\omega) C_1(y) \nabla u^N_1(y,\omega) \, dy
\nonumber
\\
&=&
\sum_{j \in \ZZ^d \cap Q_N} \frac{1}{|Q_N|} \int_{Q_N} X_j(\omega) \1_{Q+j} C_1 \EE[X_0] \nabla \overline{u_1}
\nonumber
\\
&+& 
\sum_{k, j \in \ZZ^d \cap Q_N} \frac{1}{|Q_N|} \int_{Q_N} X_j(\omega) \1_{Q+j} C_1 \Big( X_k(\omega) - \EE[X_k] \Big) \nabla \phi_1^N(\cdot-k)
\nonumber
\\
&=&
(\EE[X_0])^2 \int_Q C_1 \nabla \overline{u_1}
\nonumber
\\
&+&
\EE[X_0] \left( \frac{1}{|Q_N|} \sum_{j \in \ZZ^d \cap Q_N} \Big( X_j(\omega) - \EE[X_j] \Big) \right) \int_Q C_1 \nabla \overline{u_1}
\nonumber
\\
&+& 
\EE[X_0] \sum_{k \in \ZZ^d \cap Q_N} \frac{1}{|Q_N|} \int_{Q_N} C_1 \Big( X_k(\omega) - \EE[X_k] \Big) \nabla \phi_1^N(\cdot-k)
\nonumber
\\
&+& 
\sum_{k, j \in \ZZ^d \cap Q_N} \frac{1}{|Q_N|} \int_{Q+j} \Big( X_j(\omega) - \EE[X_j] \Big) C_1 \Big( X_k(\omega) - \EE[X_k] \Big) \nabla \phi_1^N(\cdot-k).
\nonumber
\\
\label{eq:cond2_revised_lhs}
\end{eqnarray}
In this particular (however still very generic) setting, we infer from~\eqref{eq:cond2_revised_rhs} and~\eqref{eq:cond2_revised_lhs} that Condition~\eqref{eq:cond2} reads as  
\begin{multline}
\label{eq:cond2_ter}
\frac{1}{|Q_N|} \sum_{k, j \in Q_N \cap \ZZ^d} \Big( X_k(\omega) - \EE[X_k] \Big) \Big( X_j(\omega) - \EE[X_j] \Big) I_{k,j}^N
\\
+
\frac{1}{|Q_N|} \EE[X_0] \sum_{k \in Q_N \cap \ZZ^d} \Big( X_k(\omega) - \EE[X_k] \Big) I_k^N
= 
\sum_{k \in \ZZ^d} \Cov(X_0,X_k) I_k^\infty,
\end{multline}
where 
\begin{eqnarray}
\label{eq:def_I_k}
I_k^\infty &=& \int_{Q+k}C_1(y) \nabla \phi_1(y),
\\
\label{eq:def_I_kj^N}
I_{k,j}^N &=& \int_{Q+j} C_1(y) \nabla \phi_1^N(y-k)dy,
\\
\label{eq:def_I_k^N}
I_k^N &=& \int_{Q_N} C_1(y) \nabla \phi_1^N(y-k)dy + \int_Q C_1(y) \nabla \overline{u_1}(y) dy.
%\nonumber
\end{eqnarray} 

\subsubsection{Summary}
\label{sec:summary}

In the prototypical case where
$$
A(x, \omega) = C_0 + \chi(x, \omega) C_1(x),
$$
where $C_0$ is constant, $C_1$ is $\ZZ^d$ periodic and $\chi$ takes the form~\eqref{eq:chi_form} (and where we consider the periodic approximation~\eqref{eq:correcteur-random-N} of~\eqref{eq:correcteur-complet}), we have that:
\begin{itemize}
\item The condition~\eqref{eq:cond0} (SQS condition of order 0) is systematically fulfilled.
\item In view of~\eqref{eq:cond1_bis} and~\eqref{eq:chi_form}, the condition~\eqref{eq:cond1} (SQS condition of order 1) rewrites as 
\begin{equation}
\label{eq:cond1_fin}
\frac{1}{|Q_N|} \sum_{k \in \ZZ^d \cap Q_N} X_k(\omega) = \EE\left[ X_0 \right].
\end{equation}
\item In view of~\eqref{eq:cond2_ter}, the condition~\eqref{eq:cond2} (SQS condition of order 2) writes as
\begin{multline}
\label{eq:cond2_fin}
\frac{1}{|Q_N|} \sum_{k, j \in Q_N \cap \ZZ^d} \overline{X}_k(\omega) \overline{X}_j(\omega) I_{k,j}^N
\\
+
\frac{1}{|Q_N|} \EE[X_0] \sum_{k \in Q_N \cap \ZZ^d} \overline{X}_k(\omega) I_k^N
= 
\sum_{k \in \ZZ^d} \Cov(X_0,X_k) I_k^\infty,
\end{multline}
where $\overline{X}_k(\omega) = X_k(\omega) - \EE[X_k]$.
\end{itemize}
The conditions~\eqref{eq:cond1_fin} and~\eqref{eq:cond2_fin} are henceforth called the SQS~1 and SQS~2 conditions, respectively. 

\begin{remark}
If~\eqref{eq:cond1_fin} is satisfied, then the coefficient $I_k^N$ in~\eqref{eq:cond2_fin} can be replaced by 
$$
%\begin{equation}
%\label{eq:def_I_k^N_over}
\overline{I}_k^N = \int_{Q_N} C_1(y) \nabla \phi_1^N(y-k)dy
%\end{equation} 
$$
and there is no need to compute $\overline{u_1}$.
\end{remark}

\subsection{Selection Monte Carlo sampling}
\label{sec:algo}

We are now in position to describe the selection Monte Carlo sampling we employ. We recall that the classical Monte Carlo sampling reads as follows:

\begin{algorithm}[{\bf Classical Monte Carlo}] \ \\
\label{algo:Monte_Carlo}
For $m = 1, \dots, M$,
\begin{enumerate}
\item Generate a random environment $\omega_m$.
\item Solve the truncated corrector problem~\eqref{eq:correcteur-random-N}.
\item Compute $A^\star_N(\omega_m)$.
\end{enumerate}
Compute the approximation $\dis {\cal I}^M_{MC} = \frac{1}{M} \sum_{m=1}^M A^\star_N(\omega_m)$ of $A^\star$.
\end{algorithm}

In contrast, our selection Monte Carlo sampling algorithm, in the particular case described in Section~\ref{sec:summary}, reads as follows:

\begin{algorithm} \ \\
\label{algo:SQS_tolerance}
The algorithm requires a tolerance ${\rm tol} > 0$, fixed by the user.
\begin{enumerate}
\item \label{algo:tolerance_offline} {\bf Offline stage}
\begin{enumerate}
\item Solve the equation~\eqref{eq:def_u1_k}.
\item \label{lab:toto} Compute $(I_k^\infty)_{k \in \ZZ^d}$ defined by~\eqref{eq:def_I_k}.
\item Compute the right-hand side of the SQS conditions~\eqref{eq:cond1_fin} and~\eqref{eq:cond2_fin}.
\item Solve the equations~\eqref{eq:def_u1_bar} and~\eqref{eq:def_u1^N_k}.
\item Compute $(I_{k,j}^N)_{k,j \in \ZZ^d \cap Q_N}$ and $(I_k^N)_{k \in \ZZ^d \cap Q_N}$ defined by~\eqref{eq:def_I_kj^N} and~\eqref{eq:def_I_k^N}.
\end{enumerate}
\item {\bf Online stage}\\
For $m = 1, \dots, M$,
\begin{enumerate}
\item \label{SQS_tolerance:step1} Generate a random environment $\omega_m$. 
\item Using $I_{k,j}^N$ and $I_k^N$, compute the left-hand sides of~\eqref{eq:cond1_fin} and~\eqref{eq:cond2_fin}.
\item If the left-hand sides differ from the right-hand sides by more than {\rm tol}, return to Step~\ref{SQS_tolerance:step1}.
\item Solve the truncated corrector problem~\eqref{eq:correcteur-random-N}.
\item Compute $A^\star_N(\omega_m)$.
\end{enumerate}
\end{enumerate}
Compute the approximation $\dis {\cal I}^M_{SQS} = \frac1M \sum_{m=1}^M A^\star_N(\omega_m)$ of $A^\star$.
\end{algorithm}

\begin{remark}
As pointed out above, the series in $k \in \ZZ^d$ in the right-hand side of~\eqref{eq:cond2_fin} may be truncated in an efficient manner because of the rapid decay at infinity of $\nabla \phi_1$. Therefore only a few factors $I_k^\infty$ have to be computed at Step~\ref{lab:toto}.
\end{remark}

\begin{remark}
%\label{rem:poids}
When several SQS conditions (in practice SQS~1 and SQS~2) have to be simultaneously satisfied, we simply add them up using some weighting parameter. We have not observed any particular sensitivity of our numerical results (collected in Section~\ref{sec:numerics} below) with respect to the adjustment of this parameter, provided it remains not too close to 0 and 1. %stays in the range $[0.2, 0.7]$.
\end{remark}

We have already mentioned that, in many situations, there might not be {\em any} random environments that satisfy some, or all, of the SQS conditions~\eqref{eq:cond0}--\eqref{eq:cond1}--\eqref{eq:cond2} we wish to enforce. Therefore, some adaptation is in order, and we have used in Algorithm~\ref{algo:SQS_tolerance} a tolerance parameter ${\rm tol} > 0$ for the SQS conditions to be satisfied.

However, if these conditions are enforced within some given tolerance as in Algorithm~\ref{algo:SQS_tolerance}, the following issue arises. Since the motivation for precisely considering the SQS conditions is that they are fulfilled {\em asymptotically}, the larger the truncated computational domain we consider (that is, the larger $N$), the less restrictive the conditions are, and therefore the less effective the variance reduction is likely to be. To circumvent this difficulty, a first possibility is to consider a tolerance that decreases when the size of $Q_N$ increases. We consider this variant in our theoretical study of Section~\ref{sec:zero_d} below (see formula~\eqref{eq:var_0D_loose}). More precisely, we require in Proposition~\ref{prop:analysis_imperfect} that
$$
\text{the SQS condition is satisfied with the tolerance $\frac{\lambda}{\sqrt{|Q_N|}}$}
$$
for some $\lambda$. In practice, implementing such a threshold is not an easy matter, as the rate and the constants need to be adequately adjusted. In order to avoid such technicalities, we prefer to take a slightly different perspective, the purpose of which is to always select a {\em fixed proportion} of the original sample of the ${\cal M}$ environments drawn. Practically, we pick the $M$ configurations that best satisfy the SQS conditions among the ${\cal M}$ configurations that have been drawn. 

The practical algorithm we employ is therefore as follows:

\begin{algorithm}[{\bf Selection Monte Carlo sampling}] \ \\
\label{algo:SQS_selection}
The algorithm requires a number of trials ${\cal M}$, fixed by the user.
\begin{enumerate}
\item \label{SQS_selection:precompute}{\bf Offline stage 1}: same as the offline stage of Algorithm~\ref{algo:SQS_tolerance}.
\item \label{SQS_selection:selection}{\bf Offline stage 2: selection step}\\
For $m = 1, \dots, {\cal M}$,
\begin{enumerate}
\item \label{SQS_selection:step1} Generate a random environment $\omega_m$.
\item Using $I_{k,j}^N$ and $I_k^N$, compute the left-hand sides of~\eqref{eq:cond1_fin} and~\eqref{eq:cond2_fin}.
\item Compute the error ${\rm error}_m$ between the left-hand sides and the right-hand sides of~\eqref{eq:cond1_fin} and~\eqref{eq:cond2_fin}.
\end{enumerate}
Sort the random environments $(\omega_m)_{1 \leq m \leq {\cal M}}$ according to ${\rm error}_m$. Keep the $M$ best realizations, and reject the others.
\item {\bf Online stage: resolution}\\
For $m = 1, \dots, M$,
\begin{enumerate}
\item Solve the truncated corrector problem~\eqref{eq:correcteur-random-N}.
\item Compute $A^\star_N(\omega_m)$.
\end{enumerate}
\end{enumerate}
Compute the approximation $\dis {\cal I}^M_{SQS} = \frac{1}{M} \sum_{m=1}^M A^\star_N(\omega_m)$ of $A^\star$.
\end{algorithm}

We wish to make a couple of comments about this selection Monte Carlo approach.

In full generality, the cost of Monte Carlo approaches is usually dominated by the cost of draws, and therefore selection algorithms are targeted to reject as few draws as possible.

In the present context, where boundary value problems such as~\eqref{eq:correcteur-random-N} are to be solved repeatedly, the cost of draws for the environment is negligible in front of the cost of the solution procedure for such boundary value problems. Likewise, evaluating the quantities present in e.g.~\eqref{eq:cond2_fin} is not expensive. Therefore, the purpose of the selection mechanism is to limit the number of boundary value problems to be solved, even though this comes at the (tiny) price of rejecting many environments. This also explains why we employ a simplistic rejection procedure for the selection, while in other situations of Monte Carlo samplings, one would invest in a more clever selection procedure.

A second observation is that, as potentially for any selection procedure, our selection introduces a bias (i.e. a modification of the systematic error in~\eqref{eq:error-decomposition}). The point is to ensure that the gain in variance superseeds the bias introduced by the variance reduction approach. 

\medskip

Our next section addresses some theoretical aspects of our approach. 

\section{Elements of theoretical analysis}
\label{sec:theorique}

This section contains some elements of analysis that we are able to provide. We begin with a (somewhat) general result of convergence, and next, in some simplified cases, study our approach more thoroughly.

\subsection{Proof of convergence of the approach}
\label{sec:consistency}

Formally, our approach consists in replacing an empirical average provided by the classical Monte Carlo approach to compute $\EE [A^\star_N]$ by an empirical average {\em restricted} to some environments within $Q_N$ satisfying some additional condition(s) (see Section~\ref{sec:algo}). We work at a fixed size $N$ of the truncation domain $Q_N$ and recall that $A^\star_N(\omega)$ is defined by~\eqref{eq:AstarN}. Mathematically, our approach amounts to considering conditional expectations of the type $\EE[A^\star_N \ | \ \text{SQS}]$, where $\text{SQS}$ encodes that one, or several, of the conditions summarized in~\eqref{eq:cond1_fin}--\eqref{eq:cond2_fin} are satisfied.

The least we can expect from our approach is that it converges to the correct limit when $N \to \infty$, namely $A^\star$, as in~\eqref{eq:BP_limit}. 

The theorem we now state establishes this fact. In order to prove it, we need to make some assumptions on our setting (see the details below), and also to make specific the SQS conditions we use. In Theorem~\ref{thm:convergence} below, we specifically use the SQS~1 condition, in the form~\eqref{eq:cond1_fin}.

In order to state a result as general as possible, we therefore consider a condition that reads $\dps \frac{1}{|Q_N|} \sum_{k \in \ZZ^d \cap Q_N} f(X_k) = \EE[f(X_0)]$ for some function $f$. In practice, our specific SQS~1 condition~\eqref{eq:cond1_fin} corresponds to the choice $f(x) = x$.

\begin{theorem}
\label{thm:convergence}
Let $(X_k)_{k \in \ZZ^d}$ be a sequence of independent and identically distributed scalar random variables following a common law $\mu$. We assume that $\mu$ is absolutely continuous with respect to the Lebesgue measure on $\RR$, and that, for any $k \in \ZZ^d$, $X_k(\omega) \in [-1,1]$ almost surely. We consider the stationary random field 
$$
A(y, \omega) = C_0 + \sum_{k \in \ZZ^d} X_k(\omega) \, \1_{Q+k}(y) \, C_1(y),
$$ 
where $C_0$ is constant and $C_1$ is $\ZZ^d$-periodic and bounded. We also assume that $C_0+C_1(y)$ and $C_0-C_1(y)$ are uniformly coercive, and that $C_0$ and $C_1$ are symmetric.

Let $f:\RR \mapsto \RR$ be a measurable function with compact level sets. We assume that $f$ is not constant. Then we have
\begin{equation}
\label{eq:notre_resu}
\EE\left[A^\star_N \ \Big| \ \frac{1}{|Q_N|} \sum_{k \in Q_N \cap \ZZ^d} f(X_k) = \EE[f(X_0)] \right] \xrightarrow[N \to \infty]{} A^\star,
\end{equation}
where $A^\star_N(\omega)$ is defined by~\eqref{eq:AstarN} and $A^\star$ is defined by~\eqref{eq:matrice-homogeneisee}.
\end{theorem}

Some remarks are in order.

\begin{remark}
As is the case throughout this article, we have considered the {\em periodic} approximation~\eqref{eq:correcteur-random-N} of~\eqref{eq:correcteur-complet}. The proof of Theorem~\ref{thm:convergence} actually carries over to the case of Neumann or Dirichlet boundary conditions, or any alternate truncation problem that provides some $A^{\star,N}(\omega)$ such that $A^{\star, N}_{\rm Neu}(\omega) \leq A^{\star,N}(\omega) \leq A^{\star,N}_{\rm Dir}(\omega)$ (see additional details in~\cite[Appendix]{Minvielle2015}).
\end{remark}

\begin{remark}
The assumptions regarding independence of the $X_k$, absolute continuity of their common law with respect to the Lebesgue measure and compactness of the level sets of $f$ are necessary for technical reasons, since we need to apply a general result from~\cite{BernardinOlla2014}. See below for details.
\end{remark}

The proof of Theorem~\ref{thm:convergence} is based on the following result, which is a particular case of a more general result due to C.~Bernardin and S.~Olla (see~\cite[Theorem B.2.2]{BernardinOlla2014}):

\begin{theorem}[C.~Bernardin and S.~Olla,~\cite{BernardinOlla2014}]
\label{theo:olla}
Consider $n$ scalar random variables $X_1$, \dots, $X_n$, that are independent and that all share the same probability distribution $\mu(x) \, dx$ on $\RR$. Consider a measurable function $f : \RR \mapsto \RR$, which is assumed to be not constant and to have compact level sets. Let $\dps f_0 = \EE[f(X_1)] = \int_\RR f(x) \, \mu(x) \, dx$. Consider also a bounded and continuous function $F : \RR^k \mapsto \RR$. Then
\begin{equation}
\label{eq:resu_olla}
\lim_{n \to \infty} \EE \left[ F(X_1,\dots,X_k) \ \Big| \ \frac{1}{n} \sum_{i=1}^n f(X_i) = f_0 \right] =  \EE \left[ F(X_1,\dots,X_k) \right].
\end{equation}
\end{theorem}
Note that, when $n \to \infty$, the quantity $\dps \frac{1}{n} \sum_{i=1}^n f(X_i(\omega))$ almost surely converges to $f_0$. Theorem~\ref{theo:olla} shows that conditioning on the manifold $\dps \frac{1}{n} \sum_{i=1}^n f(X_i(\omega)) = f_0$ does not change the value (when $n \to \infty$) of the expectation of a function $F$ of a {\em finite} number $k$ of random variables. Note that the condition that $k$ is independent of $n$ can be somewhat relaxed. It is indeed shown in~\cite{dembo_zeitouni_1996} that one can take $k = o(n)$ in some cases. It is also shown there that one cannot take $k=n$.   

\medskip

In our context, the variable $X_i$ is the value of the field $A$ on the cell $Q+i$. The conditioning in the left-hand side of~\eqref{eq:resu_olla} is identical to the conditioning in the left-hand side of~\eqref{eq:notre_resu}.

The difference between Theorem~\ref{theo:olla} and our result lies in the quantity of which we compute the expectation. In our case, this quantity is $A^\star_N(\omega)$, which is (asymptotically when $N \to \infty$) a function of {\em all} the variables $X_i$ and not only of a {\em finite} number of them. We hence cannot directly use Theorem~\ref{theo:olla}. The proof of our result essentially amounts to introducing an upper bound and a lower bound on $A^\star_N(\omega)$ that both read as a sum of functions that depend on a {\em finite} number of random variables (see e.g.~\eqref{eq:olla0} below). We will then be in position to apply Theorem~\ref{theo:olla} on these functions. 

\begin{proof}[Proof of Theorem~\ref{thm:convergence}]
We fix some $p \in \RR^d$. For the sake of clarity, the approximate homogenized matrix $A^\star_N(\omega)$ defined by~\eqref{eq:AstarN} is here denoted $A^{\star,N}_{\rm per}(\omega)$, to emphasize that we have considered periodic boundary conditions. Since the matrix $A$ is symmetric, we have
$$
p^T A^{\star,N}_{\rm per}(\omega) p = \inf \left\{ {\cal J}_{Q_N}(v,\omega), \ \ v \in H^1_{\rm per}(Q_N) \right\},
$$
where 
$$
{\cal J}_{Q_N}(v,\omega) = \frac{1}{|Q_N|} \int_{Q_N} (p + \nabla v)^T A(\cdot,\omega) (p + \nabla v).
$$
We have considered in~\eqref{eq:correcteur-random-N} periodic boundary conditions. As is well-known, other boundary conditions can be used, and these alternate approximations will be useful for the proof. 

\medskip

{\bf Step 1: Upper bound.} We first introduce an approximation of $A^\star$ using a truncated corrector problem complemented with homogeneous Dirichlet boundary conditions. We consider the problem
$$
%\begin{equation} 
%\label{eq:correcteur-random-N_dir}
\left\{
\begin{array}{c}
-\dive \Big( A(\cdot,\omega) \left( p +  \nabla w_{p, \rm Dir}^N(\cdot,\omega) \right) \Big) = 0 \ \ \mbox{in $Q_N$},
\\
w_{p, \rm Dir}^N(\cdot,\omega) = 0 \ \ \mbox{on $\partial Q_N$},
\end{array}
\right.
%\end{equation}
$$
which yields an approximation of $A^\star$ that we denote $A^{\star,N}_{\rm Dir}(\omega)$ and which is defined by
$$
%\begin{equation}
%\label{eq:AstarN_dir}
\forall p \in \RR^d, \quad A^{\star,N}_{\rm Dir}(\omega) p = \frac{1}{|Q_N|} \int_{Q_N} A(\cdot,\omega) (p + \nabla w_{p, \rm Dir}^N(\cdot,\omega)).
%\end{equation}
$$
As shown in~\cite{BourgeatPiatnitski2004}, we know that
\begin{equation}
\label{eq:resu_bp_dir}
\lim_{N \to \infty} A^{\star,N}_{\rm Dir}(\omega) = A^\star \quad \text{a.s.}
\end{equation}
Since $A$ is symmetric, we have
$$
p^T A^{\star,N}_{\rm Dir}(\omega) p = \inf \left\{ {\cal J}_{Q_N}(v,\omega), \ \ v \in H^1_0(Q_N) \right\}.
$$
The matrix $A^{\star,N}_{\rm Dir}(\omega)$ is always larger (in the sense of symmetric matrices) than $A^{\star,N}_{\rm per}(\omega)$. Indeed, let $v \in H^1_0(Q_N)$, and consider its $Q_N$-periodic extension $\widetilde{v}$. Then this function belongs to $H^1_{\rm per}(Q_N)$. We hence have that
$$
p^T A^{\star,N}_{\rm per}(\omega) p \leq {\cal J}_{Q_N}(\widetilde{v},\omega) = {\cal J}_{Q_N}(v,\omega). 
$$
Minimizing over $v \in H^1_0(Q_N)$, we get that
\begin{equation}
\label{eq:olla3}
p^T A^{\star,N}_{\rm per}(\omega) p \leq p^T A^{\star,N}_{\rm Dir}(\omega) p \quad \text{a.s.}
\end{equation}
Just as $A^{\star,N}_{\rm per}(\omega)$, the matrix $A^{\star,N}_{\rm Dir}(\omega)$ depends on all the random variables $X_i(\omega)$, $i \in Q_N \cap \ZZ^d$. But, thanks to the use of homogeneous Dirichlet boundary conditions, it can be bounded from above by a sum of matrices that depend only on a finite number of random variables. To show this, we proceed as follows. 

For any positive integers $N$ and $R$, we introduce the integer part $M$ of $N/R$. Then $Q_N$ can decomposed into a set of cubes of size $R^d$, up to some boundary layer $B_{N,R}$: 
$$
Q_N = \Big( \cup_{j \in \ZZ^d, \, |j| \leq M} R \,j + Q_R \Big) \cup B_{N,R}.
$$

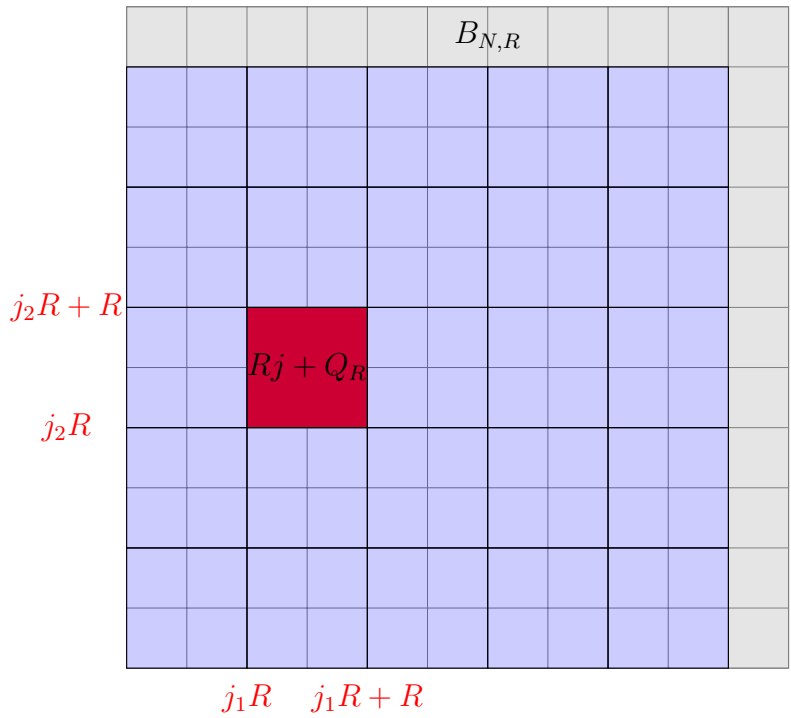
\begin{figure}[htbp]
\hspace{0.3cm}
%  \centering
  \def\Lnet{0.8}    
  \begin{tikzpicture}                               
    \draw[style=help lines] (0,0) grid[step=\Lnet] (11*\Lnet,11*\Lnet);
    
        \filldraw[fill=gray, fill opacity=0.2, draw=gray] (0,10*\Lnet)
        rectangle (10*\Lnet,11*\Lnet);  % gray rectangle 1
    \filldraw[fill=gray, fill opacity=0.2, draw=gray] (10*\Lnet,0)
        rectangle (11*\Lnet,11*\Lnet);  % gray rectangle 2
           \draw [black](6*\Lnet,10.5*\Lnet) node [] {$B_{N,R}$}; %%%%%% 

        \filldraw[fill=red, fill opacity=1, draw=gray] (2*\Lnet,4*\Lnet)
        rectangle (4*\Lnet,6*\Lnet);  % Rj + Q_R

    %  Draws the dashed grid
    \foreach \x in {0,1,...,4} % Two indices running over each
    {
      \foreach \y in {0,1,...,4}
      {
          \filldraw[fill=blue, fill opacity=0.2, draw=black] (2*\x*\Lnet,2*\y*\Lnet) %(-\Lnet*0.5,-\Lnet*0.5)
        rectangle (2*\x*\Lnet+2*\Lnet,2*\y*\Lnet+2*\Lnet);  % Each supercell jR + Q_R
      }
    }
   \draw [red](2*\Lnet,-0.5*\Lnet) node [] {$j_1R$}; %%%%%%  j_1R
   \draw [red] (4*\Lnet,-0.5*\Lnet) node [] {$j_1R + R$}; %%%%%% j_1R + R
    \draw [red] (-\Lnet,4*\Lnet) node [] {$j_2R$}; %%%%%%  j_2R
    \draw [red] (-\Lnet,6*\Lnet) node [] {$j_2R+R$}; %%%%%% j_2R+R
 
 \draw [black] (3*\Lnet, 5*\Lnet) node {$Rj + Q_R$}; %%%%% Label of the Rj + Q_R cell

\end{tikzpicture}
\caption{The domain $Q_N$ (here represented for $N=11$) is split into domains of size $R^d$ (here $R=2$; one of them is shown in red on the figure), up to some boundary layer $B_{N,R}$ (shown in light gray). \label{fig:decompo}}
\end{figure}

For any $j \in \ZZ^d$, $|j| \leq M$, consider a function $v_j \in H^1_0(R \,j + Q_R)$. We now define the function $v$ on $Q_N$ as:
\begin{itemize}
\item for any $x \in R \,j + Q_R$, we set $v(x) = v_j(x)$;
\item if $x \in B_{N,R}$, we set $v(x) = 0$.
\end{itemize}
The function $v$ belongs to $H^1_0(Q_N)$. We hence write that
$$
p^T A^{\star,N}_{\rm Dir}(\omega) p \leq {\cal J}_{Q_N}(v,\omega) = \frac{|Q_R|}{|Q_N|} \sum_{j \in \ZZ^d, \, |j| \leq M} {\cal J}_{R \,j + Q_R}(v_j,\omega).
$$
Minimizing over the functions $v_j \in H^1_0(R \,j + Q_R)$, we hence get that
\begin{equation}
\label{eq:olla0}
p^T A^{\star,N}_{\rm Dir}(\omega) p \leq \frac{|Q_R|}{|Q_N|} \sum_{j \in \ZZ^d, \, |j| \leq M} Y_j(\omega) \quad \text{a.s.}
\end{equation}
where
$$
Y_j(\omega) = \inf \left\{ {\cal J}_{R \,j + Q_R}(v,\omega), \ \ v \in H^1_0(R \,j + Q_R) \right\}.
$$
Since $A$ is stationary, we note that all the random variables $Y_j(\omega)$ share the same law. Moreover, we observe that $Y_0(\omega) = p^T A^{\star,R}_{\rm Dir}(\omega) p$, which is the approximation of the homogenized matrix using Dirichlet boundary conditions on $Q_R$.

We now take the conditional expectation of~\eqref{eq:olla0}, and use the fact that the variables $Y_j$ all share the same law:
\begin{multline*}
\EE\left[p^T A^{\star,N}_{\rm Dir}(\omega) p \ \Big| \ \frac{1}{|Q_N|} \sum_{k \in Q_N \cap \ZZ^d} f(X_k) = \EE[f(X_0)] \right]
\\
\leq
\frac{R^d \, M^d}{N^d} \ \EE\left[ p^T A^{\star,R}_{\rm Dir}(\omega) p \ \Big| \ \frac{1}{|Q_N|} \sum_{k \in Q_N \cap \ZZ^d} f(X_k) = \EE[f(X_0)] \right].
\end{multline*}
We next observe that $p^T A^{\star,R}_{\rm Dir}(\omega) p$ only depends on a {\em finite} number of random variables, namely only on $X_k(\omega)$ with $k \in Q_R \cap \ZZ^d$. We are thus in position to use Theorem~\ref{theo:olla}, which yields the limit of the above right-hand side when $N \to \infty$. Hence, for any fixed $R$, we have
\begin{multline*}
\limsup_{N \to \infty} \, \EE\left[p^T A^{\star,N}_{\rm Dir}(\omega) p \ \Big| \ \frac{1}{|Q_N|} \sum_{k \in Q_N \cap \ZZ^d} f(X_k) = \EE[f(X_0)] \right]
\\
\leq
\EE\left[ p^T A^{\star,R}_{\rm Dir}(\omega) p \right].
\end{multline*}
Letting $R$ go to $\infty$ in the above bound and using~\eqref{eq:resu_bp_dir}, we obtain that
$$
\limsup_{N \to \infty} \, \EE\left[p^T A^{\star,N}_{\rm Dir}(\omega) p \ \Big| \ \frac{1}{|Q_N|} \sum_{k \in Q_N \cap \ZZ^d} f(X_k) = \EE[f(X_0)] \right]
\leq
p^T A^\star p.
$$
Using~\eqref{eq:olla3}, we deduce that
\begin{equation}
\label{eq:step1}
\forall p \in \RR^d, \quad
\limsup_{N \to \infty} \, p^T U_N p \leq p^T A^\star p,
\end{equation}
where
\begin{equation}
\label{eq:olla5}
U_N = \EE\left[A^{\star,N}_{\rm per}(\omega) \ \Big| \ \frac{1}{|Q_N|} \sum_{k \in Q_N \cap \ZZ^d} f(X_k) = \EE[f(X_0)] \right].
\end{equation}

\medskip

{\bf Step 2: Lower bound.} We now introduce an approximation of $A^\star$ using a truncated problem complemented with Neumann boundary conditions. We consider the problem
\begin{equation} 
\label{eq:correcteur-random-N_neu}
\left\{
\begin{array}{c}
\dps -\dive \Big( A(\cdot,\omega) \left( p +  \nabla w_{p, \rm Neu}^N(\cdot,\omega) \right) \Big) = 0 \ \ \mbox{in $Q_N$},
\\
\dps n^T A(\cdot,\omega) (p + \nabla w_{p, \rm Neu}^N(\cdot,\omega)) = n^T p \ \ \mbox{on $\partial Q_N$},
\end{array}
\right.
\end{equation}
which yields an approximation of $A^\star$ that we denote $A^{\star,N}_{\rm Neu}(\omega)$ and which is defined by
\begin{equation}
\label{eq:AstarN_neu_pre}
A^{\star,N}_{\rm Neu}(\omega) = \left( S^{\star,N}_{\rm Neu}(\omega) \right)^{-1},
\end{equation}
where $S^{\star,N}_{\rm Neu}(\omega)$ is defined by
\begin{equation}
\label{eq:AstarN_neu}
\forall p \in \RR^d, \quad S^{\star,N}_{\rm Neu}(\omega) p = \frac{1}{|Q_N|} \int_{Q_N} p + \nabla w_{p, \rm Neu}^N(\cdot,\omega).
\end{equation} 
We refer to Remark~\ref{rem:heur} below for some heuristic justification of~\eqref{eq:AstarN_neu_pre}--\eqref{eq:AstarN_neu}.

As recalled in~\cite[Appendix]{Minvielle2015}, we have that
\begin{equation}
\label{eq:resu_bp_neu}
\lim_{N \to \infty} A^{\star,N}_{\rm Neu}(\omega) = A^\star \quad \text{a.s.}
\end{equation}
and
\begin{equation}
\label{eq:olla6}
p^T A^{\star,N}_{\rm Neu}(\omega) p \leq p^T A^{\star,N}_{\rm per}(\omega) p \quad \text{a.s.}
\end{equation}
In addition, we have the following variational characterization:
\begin{equation}
\label{eq:olla1}
p^T S^{\star,N}_{\rm Neu}(\omega) p = \inf \left\{ {\cal E}_{Q_N}(\sigma,\omega), \ \ \sigma \in V(Q_N) \right\},
\end{equation}
where 
$$
{\cal E}_{Q_N}(\sigma,\omega) = \frac{1}{|Q_N|} \int_{Q_N} (p+\sigma)^T A^{-1}(\cdot,\omega) (p+\sigma)
$$
and
$$
V(Q_N) = \left\{ \sigma \in (L^2(Q_N))^d, \ \ \dive \sigma = 0 \text{ in $Q_N$}, \ \ n^T \sigma = 0 \text{ on $\partial Q_N$} \right\}.
$$
The matrix $S^{\star,N}_{\rm Neu}(\omega)$ (and hence the matrix $A^{\star,N}_{\rm Neu}(\omega)$) depends on all the variables $X_i(\omega)$, $i \in Q_N \cap \ZZ^d$. However, thanks to the characterization~\eqref{eq:olla1}, it can be bounded from above by a sum of matrices that depend only on a finite number of random variables. 

To show this, we proceed as in Step 1 of the proof. For any positive integers $N$ and $R$, we introduce the integer part $M$ of $N/R$, and decompose $Q_N$ into a set of cubes of size $R^d$, up to some boundary layer $B_{N,R}$ (see Figure~\ref{fig:decompo}):
$$
Q_N = \Big( \cup_{j \in \ZZ^d, \, |j| \leq M} R \,j + Q_R \Big) \cup B_{N,R}.
$$
For any $j \in \ZZ^d$, $|j| \leq M$, consider a function $\sigma_j \in V(R \,j + Q_R)$. We now define the function $\sigma$ on $Q_N$ as:
\begin{itemize}
\item for any $x \in R \,j + Q_R$, we set $\sigma(x) = \sigma_j(x)$;
\item if $x \in B_{N,R}$, we set $\sigma(x) = 0$.
\end{itemize}
We claim that $\sigma \in V(Q_N)$. We indeed first have that $\sigma \in (L^2(Q_N))^d$. We next consider $\varphi \in C^\infty_0(Q_N)$ and compute that
\begin{eqnarray*}
\langle \dive \sigma, \varphi \rangle 
&=& 
- \langle \sigma, \nabla \varphi \rangle 
\\
&=&
- \sum_{j \in \ZZ^d, \, |j| \leq M} \int_{R \,j + Q_R} \sigma_j \cdot \nabla \varphi
\\
&=&
- \sum_{j \in \ZZ^d, \, |j| \leq M} \int_{\partial(R \,j + Q_R)} n^T_j \sigma_j \ \varphi
\\
&=& 0,
\end{eqnarray*}
where $n_j$ is the outward normal to the domain $R \,j + Q_R$. We hence have checked that $\sigma \in V(Q_N)$.

We next write that
$$
p^T S^{\star,N}_{\rm Neu}(\omega) p \leq {\cal E}_{Q_N}(\sigma,\omega) = \frac{|Q_R|}{|Q_N|} \sum_{j \in \ZZ^d, \, |j| \leq M} {\cal E}_{R \,j + Q_R}(\sigma_j,\omega).
$$
Minimizing over the functions $\sigma_j \in V(R \,j + Q_R)$, we hence get that
\begin{equation}
\label{eq:olla2}
p^T S^{\star,N}_{\rm Neu}(\omega) p \leq \frac{|Q_R|}{|Q_N|} \sum_{j \in \ZZ^d, \, |j| \leq M} Z_j(\omega) \quad \text{a.s.}
\end{equation}
where
$$
Z_j(\omega) = \inf \left\{ {\cal E}_{R \,j + Q_R}(\sigma,\omega), \ \ \sigma \in V(R \,j + Q_R) \right\}.
$$
Since $A$ is stationary, we note that all the random variables $Z_j(\omega)$ share the same law. Moreover, we observe that $Z_0(\omega) = p^T S^{\star,R}_{\rm Neu}(\omega) p$.

We now take the conditional expectation of~\eqref{eq:olla2}, and use the fact that the variables $Z_j$ all share the same law:
\begin{multline*}
\EE\left[p^T S^{\star,N}_{\rm Neu}(\omega) p \ \Big| \ \frac{1}{|Q_N|} \sum_{k \in Q_N \cap \ZZ^d} f(X_k) = \EE[f(X_0)] \right]
\\
\leq
\frac{R^d \, M^d}{N^d} \EE\left[ p^T S^{\star,R}_{\rm Neu}(\omega) p \ \Big| \ \frac{1}{|Q_N|} \sum_{k \in Q_N \cap \ZZ^d} f(X_k) = \EE[f(X_0)] \right].
\end{multline*}
We observe that $p^T S^{\star,R}_{\rm Neu}(\omega) p$ only depends on a {\em finite} number of random variables, namely only on $X_k$ with $k \in Q_R \cap \ZZ^d$. We are thus in position to use Theorem~\ref{theo:olla}, which yields the limit of the above right-hand side when $N \to \infty$. Hence, for any fixed $R$, we have
\begin{multline*}
\limsup_{N \to \infty} \, \EE\left[p^T S^{\star,N}_{\rm Neu}(\omega) p \ \Big| \ \frac{1}{|Q_N|} \sum_{k \in Q_N \cap \ZZ^d} f(X_k) = \EE[f(X_0)] \right]
\\
\leq
\EE\left[ p^T S^{\star,R}_{\rm Neu}(\omega) p \right].
\end{multline*}
Letting $R$ go to $\infty$ in the above bound and using~\eqref{eq:AstarN_neu_pre} and~\eqref{eq:resu_bp_neu}, we obtain that
\begin{multline*}
\limsup_{N \to \infty} \, \EE\left[p^T S^{\star,N}_{\rm Neu}(\omega) p \ \Big| \ \frac{1}{|Q_N|} \sum_{k \in Q_N \cap \ZZ^d} f(X_k) = \EE[f(X_0)] \right]
\\
\leq
p^T (A^\star)^{-1} p.
\end{multline*}
Using~\eqref{eq:AstarN_neu_pre} and~\eqref{eq:olla6}, we deduce that
\begin{multline*}
\limsup_{N \to \infty} \, \EE\left[p^T \left( A^{\star,N}_{\rm per}(\omega) \right)^{-1} p \ \Big| \ \frac{1}{|Q_N|} \sum_{k \in Q_N \cap \ZZ^d} f(X_k) = \EE[f(X_0)] \right]
\\
\leq
p^T (A^\star)^{-1} p.
\end{multline*}
Using Jensen inequality, we infer from the above bound that
\begin{equation}
\label{eq:step2}
\forall p \in \RR^d, \quad
\limsup_{N \to \infty} \, p^T \left( U_N \right)^{-1} p \leq p^T (A^\star)^{-1} p,
\end{equation}
where the matrix $U_N$ is defined by~\eqref{eq:olla5}. 

\medskip

{\bf Step 3: Conclusion.} We eventually show that~\eqref{eq:step1} and~\eqref{eq:step2} imply that $U_N$ converges to $A^\star$ when $N \to \infty$.

From the assumptions on $A$, we know that there exists $0 <a_- \leq a_+ < \infty$ such that, for any $N$ and almost surely, $a_- \leq A^{\star,N}_{\rm per}(\omega) \leq a_+$. Hence, for any $N$, the symmetric matrix $U_N$ satisfies $a_- \leq U_N \leq a_+$. We can thus extract a subsequence $U_{\varphi(N)}$ that converges to some symmetrix matrix $B$. Let us show that $B = A^\star$. 

Let $p \in \RR^d$. We first observe that, by definition,
$$
\limsup_{k \to \infty} \ p^T U_k p \geq \lim_{k \to \infty} \ p^T U_{\varphi(k)} p = p^T B p.
$$
We thus infer from~\eqref{eq:step1} that
\begin{equation}
\label{eq:step1_bis}
\forall p \in \RR^d, \quad p^T B p \leq p^T A^\star p.
\end{equation}
We now proceed likewise with $U_k^{-1}$. We observe that, 
$$
\limsup_{k \to \infty} \ p^T U_k^{-1} p \geq \lim_{k \to \infty} \ p^T U_{\varphi(k)}^{-1} p = p^T B^{-1} p.
$$
We thus infer from~\eqref{eq:step2} that
\begin{equation}
\label{eq:step2_bis}
\forall p \in \RR^d, \quad p^T B^{-1} p \leq p^T (A^\star)^{-1} p.
\end{equation}
Collecting~\eqref{eq:step1_bis} and~\eqref{eq:step2_bis}, we deduce that $B = A^\star$. 

The sequence $U_N$ is bounded, and we have shown that any converging subsequence converges to $A^\star$. This implies that $U_N$ converges to $A^\star$ when $N \to \infty$, which is exactly the result~\eqref{eq:notre_resu}. This concludes the proof of Theorem~\ref{thm:convergence}.
\end{proof}

\begin{remark}
\label{rem:heur}
In view of~\eqref{eq:correcteur-random-N_neu}, we can check that
$$
\frac{1}{|Q_N|} \int_{Q_N} A(\cdot,\omega) \left( p +  \nabla w_{p, \rm Neu}^N(\cdot,\omega) \right) = p.
$$
The definition~\eqref{eq:AstarN_neu_pre}--\eqref{eq:AstarN_neu} can hence be understood as 
$$ 
\Big\langle A(\cdot,\omega) \left( p +  \nabla w_{p, \rm Neu}^N(\cdot,\omega) \right) \Big\rangle = A^{\star,N}_{\rm Neu}(\omega) \ \Big\langle p +  \nabla w_{p, \rm Neu}^N(\cdot,\omega) \Big\rangle,
$$
where $\dps \langle \cdot \rangle = |Q_N|^{-1} \int_{Q_N} \cdot$ is the average on $Q_N$.
\end{remark}

\subsection{Complete analysis in some simple cases}
\label{sec:theorique_degrade}

In this section, we aim at improving the convergence result~\eqref{eq:notre_resu} of the previous section by quantifying both the statistical and systematic errors, in order to assess the efficiency of our approach. We are only able to proceed in simple situations where all the quantities are indeed accessible using analytic calculations. These two situations are examined in Sections~\ref{sec:zero_d} and~\ref{sec:one_d} respectively. For the sake of brevity, and because the proofs are not very enlightening and are not likely to carry over to more general cases, we do not provide the proofs of our claims here. We refer to~\cite{Minvielle2015} where they are presented in details.

We establish below that our approach preserves the {\em rate} of decay of the standard Monte Carlo sampling both for the systematic and the statistical error (and thus, in particular, the systematic error remains, in rate, smaller than the statistical error). Furthermore, the {\em prefactor} in the statistical error is significantly reduced by our approach.

\subsubsection{``Zero-dimensional'' homogenization}
\label{sec:zero_d}

As simplest possible situation, we consider a function $g: \RR \mapsto \RR$ and the random variables $(X_i)_{1 \leq i \leq n}$. We assume that these random variables are independent and that they are all centered Gaussian random variables with unit variance. We also assume that $g \in C^1(\RR)$ and that $\EE \left[ |g(X_1)| + |g'(X_1)| \right] < \infty$. Note that it is not surprising to make some smoothness assumptions on $g$ as we are here after {\em rates} of convergence, and not only a convergence result as in Section~\ref{sec:consistency}. 

We set 
$$
\xi: x \mapsto \frac{1}{n} \sum_{i=1}^n x_i.
$$ 
Assume we want to compute $\EE[g(X_1)]$. A classical Monte Carlo approach would approximate this by the limit of the empirical mean $\dis \lim_{n \to \infty} \frac{1}{n} \sum_{i=1}^n g(X_i(\omega))$. In this particular instance, the simplest version of our variance reduction approach instead considers $\dps \lim_{n \to \infty} \frac{1}{n} \sum_{i = 1}^n g(X_i(\omega))$ for realizations $X(\omega)$ that satisfy $\xi(X(\omega)) = 0$.

In this simple case, the bias of the classical approach is actually identically zero: of course, $\dis \EE\left[ \frac{1}{n} \sum_{i=1}^n g(X_i) \right]$ does not depend on $n$. The statistical error is controlled by the Central Limit Theorem and is asymptotically of order $\dis \sqrt{\frac{\Var[g(X_1)]}{n}}$.

\begin{proposition}
\label{prop:analysis_0D}
Under the assumptions of this section, the bias of the selection method is of order $1/n$. More specifically, 
\begin{equation}
\label{eq:bias_0D}
\EE\left[ \frac{1}{n} \sum_{i=1}^n g(X_i) \ \Big| \ \xi(X)=0 \right] - \EE[g(X_1)] = - \frac1{2n} \EE[g'(X_1)] + O\left(\frac{1}{n^2}\right).
\end{equation}
The variance of the selection method is reduced by a factor asymptotically independent of $n$. More specifically,
\begin{equation}
\label{eq:var_0D}
\frac{\Var\left[ \frac{1}{n} \sum_{i=1}^n g(X_i) \ \Big| \ \xi(X) = 0 \right]}
{\Var\left[ \frac{1}{n} \sum_{i=1}^n g(X_i) \right]}
= 1 - \frac{(\EE[g'(X_1)])^2}{\Var[g(X_1)]} + O\left(\frac{1}{n}\right).
\end{equation}
\end{proposition}
In view of~\eqref{eq:bias_0D}--\eqref{eq:var_0D}, we observe that, at the price of introducing a bias of order $O\left(1/n\right)$, our approach reduces the statistical error from $\dis \frac{\lambda_{\rm MC}}{\sqrt{n}}$ to $\dis \frac{\lambda_{\rm SQS}}{\sqrt{n}}$ (with $\lambda_{\rm SQS}<\lambda_{\rm MC}$), and therefore, for sufficiently large $n$, reduces the total error.

\medskip

The following result covers the case where we insert a non-zero tolerance in Algorithm~\ref{algo:SQS_tolerance}.

\begin{proposition}
\label{prop:analysis_imperfect}
Under the assumptions of this section, consider the selection method where we condition on the realizations such that $\dis \frac{z_0}{\sqrt{n}} \leq \xi(X(\omega)) \leq \frac{z_1}{\sqrt{n}}$, for some $z_0$ and $z_1 > z_0$ in $\RR$. Then, for any choice of $z_0$ and $z_1 > z_0$, the variance of the selection method is reduced by a factor asymptotically independent of $n$:
\begin{multline}
\label{eq:var_0D_loose}
\frac{\Var\left[ \frac{1}{n}\sum_{i=1}^n g(X_i) \ \Big| \ \frac{z_0}{\sqrt{n}} \leq \xi(X) \leq \frac{z_1}{\sqrt{n}} \right]}
{\Var\left[ \frac{1}{n}\sum_{i=1}^n g(X_i) \right]}
\\
= 1 - (1-C)\frac{(\EE[g'(X_1)])^2}{\Var[g(X_1)]} + O\left(\frac{1}{n}\right),
\end{multline}
where $C = \Var\left[ X_1 \ \Big| \ z_0 \leq X_1 \leq z_1 \right]$.
\end{proposition}
The conditioning $z_0/\sqrt n \leq \xi(X(\omega)) \leq z_1/\sqrt n$ is deliberately chosen in order to match the rate of the Central Limit Theorem. It corresponds to the selection of a fixed {\em proportion} of samples (as in Algorithm~\ref{algo:SQS_selection} when ${\cal M}$ is proportional to $M$). Note that $C > 0$, hence the variance is less reduced than when conditioning at $\xi(X)=0$ (which is the case considered in Proposition~\ref{prop:analysis_0D}). Note also that the variance is reduced (with respect to the standard Monte Carlo sampling) if, and only if, $1-C \geq 0$. We are yet unable to conclude that this is the case in general. We simply note that, when $z_1 = -z_0 > 0$, then $C = 1$, yielding no gain. 

\subsubsection{One-dimensional homogenization}
\label{sec:one_d}

In the one-dimensional case, the homogenization of a random field $\dis a:(y, \omega) \mapsto \sum_{i\in \ZZ} \overline{g}(X_i(\omega)) \1_{(i,i+1)}(y)$ (where $\overline{g}$ is valued, say, in $[a_-,a_+]$ with $a_->0$) is a simple harmonic average. It is readily seen that 
$$
a^\star_N(\omega) = \left( \frac{1}{N} \sum_{i=1}^N \frac{1}{\overline{g}(X_k)} \right)^{-1} = \varphi \left( \frac{1}{N} \sum_{i=1}^N \frac{1}{\overline{g}(X_k)} \right) \quad \text{with $\varphi(x) = 1/x$}.
$$
Formally, the problem is thus analogous to that of the previous section, for a certain $\varphi: \RR \mapsto \RR$ instead of $\varphi = \text{Id}$. Therefore, it is sufficient to prove consistency and variance reduction for quantities of the form $\dis \varphi\left(\frac{1}{N} \sum_{i=1}^N g(X_i)\right)$.

\begin{proposition}
\label{prop:analysis_1D}
Consider a smooth function $\varphi : \RR \mapsto \RR$. Under the assumptions of this section, the bias of the standard method and that of the selection method respectively are
\begin{equation}
\label{eq:bias_1D_MC}
\EE\left[ \varphi\left(\frac{1}{N} \sum_{i=1}^N g(X_i)\right) \right] - \varphi\left(g_0\right) = \frac{\varphi''\left(g_0\right)}{2N} \Var[g(X_1)] + O\left(\frac{1}{N^2}\right)
\end{equation}
and
\begin{multline}
\label{eq:bias_1D_SQS}
\EE\left[ \varphi\left(\frac{1}{N} \sum_{i=1}^N g(X_i)\right) \ \Big| \ \xi(X)=0 \right] - \varphi\left(g_0\right) 
\\
= \frac{\varphi''\left(g_0\right)}{2N} \left(\Var[g(X_1)] - (\EE[g'(X_1)])^2\right)
- \frac{\varphi'\left(g_0\right)}{2N}\EE[X_1 g'(X_1)]+o\left(\frac{1}{N}\right),
\end{multline}
with $\dis g_0 = \EE\left[g(X_1)\right]$.

The variance of the selection method is reduced by a factor asymptotically independent of $N$:
\begin{equation}
\label{eq:var_1D}
\frac{\Var\left[ \varphi\left(\frac{1}{N} \sum_{i=1}^N g(X_i)\right) \ \Big| \ \xi(X) = 0 \right]}
{\Var\left[ \varphi\left(\frac{1}{N} \sum_{i=1}^N g(X_i)\right) \right]}
= 1 - \frac{(\EE[g'(X_1)])^2}{\Var[g(X_1)]} + o\left(1\right).
\end{equation}
\end{proposition}

To keep things simple, we do not investigate whether a more general result, accounting for some tolerance in the manner our condition is fulfilled (in the spirit of Proposition~\ref{prop:analysis_imperfect}), holds here.

Proposition~\ref{prop:analysis_1D} shows that the bias is unchanged in rate, while the prefactor for the variance is reduced. Since the variance only decays at the rate $1/\sqrt{N}$ while the bias decays at the rate $1/N$, we see that our approach indeed reduces the total error for sufficiently large $N$.

In the numerical practice (mimicking in this one-dimensional setting what is actually performed for higher dimensional settings -- although it is in some sense unnecessary here), we generate several, independent realizations of the $N$-tuples $(X_i)_{1 \leq i \leq N}$ corresponding to as many draws of environments within the ``cube'' $Q_N$. In the classical Monte Carlo approach, we keep all such $N$-tuples. In our approach, we only consider those that satisfy an additional criterion.

An empirical mean (aimed at approximating $A^\star$) is then computed. The systematic error and the statistical error of the latter approximation are precisely related to the errors estimated in~\eqref{eq:bias_1D_MC}--\eqref{eq:bias_1D_SQS}--\eqref{eq:var_1D} respectively. Thus a theoretical assessment of our practical approach.

\section{Numerical experiments}
\label{sec:numerics}

We first present in this section some numerical experiments that show the robustness of our variance reduction approach with respect to the tolerance with which we enforce the SQS conditions (see Section~\ref{sec:selection}). We next turn to studying the performance of our approach in Section~\ref{sec:convergence}.

We consider the test-case when $A$ reads as in~\eqref{eq:A_form}, that is
$$
A_\eta(x, \omega) = C_0(x, \omega) + \eta \ \chi(x, \omega) C_1(x, \omega),
$$
with $\eta = 1/2$, $C_0 = C_1 = \text{Id}$, and $\chi$ is of the form~\eqref{eq:chi_form}, that is
$$
\chi(x, \omega) = \sum_{k \in \ZZ^d} X_k(\omega) \1_{Q+k}(x).
$$
The random variables $X_k$ are i.i.d. and follow a Bernoulli law of parameter $1/2$ valued in $\{-1, +1\}$. The contrast (i.e. the ratio of the largest value of $A$ divided by its minimum value) is equal to $3$. The influence of the contrast on the efficiency of our approach is investigated at the end of Section~\ref{sec:convergence} (see Table~\ref{tab:contrast}). We consider there much larger values of the contrast (however all smaller than 20).

In what follows, we only consider Algorithm~\ref{algo:SQS_selection}, where we take $M = 100$ and ${\cal M} = 2000$ (thus an acceptance ratio of $5 \%$).

In this setting, the SQS~1 condition as stated in~\eqref{eq:cond1_fin} is satisfied if and only if the numbers of cells within which $X_k(\omega) = 1$ is equal to the number of cells within which $X_k(\omega) = -1$. It is thus possible to enforce~\eqref{eq:cond1_fin} by randomly selecting $|Q_N|/2$ cells within the $|Q_N|$ cells that are in $Q_N$, and setting $X_k = 1$ on these cells and $X_k = -1$ on the others.

\medskip

In all our tests, we have kept the computational time fixed, or almost fixed, since the additional time needed by the selection step (namely Steps~\ref{SQS_selection:precompute} and~\ref{SQS_selection:selection} of Algorithm~\ref{algo:SQS_selection}) is roughly 5\% of the total original computational time. %Typically, for a supercell sidelength of $N=20$ and a mesh size $h=0.2$, we spend $9000$ seconds to solve 100 corrector problems. On the other hand, $500$ seconds are needed to perform the offline computations and the selection process (evaluation of the ${\cal M}$ left-hand sides of~\eqref{eq:cond2_fin} and sort of the realizations to identify the $M$ best ones).

\medskip

We conclude this section with some numerical tests on a problem involving a more general geometry of microstructures (see Section~\ref{sec:mech}). On such a problem, we again obtain a significant reduction of the variance, at no additional computational cost.

\subsection{Robustness of the approach}
\label{sec:selection}

As pointed out above, the SQS~2 condition as stated in~\eqref{eq:cond2_fin} is only enforced in Algorithm~\ref{algo:SQS_selection} up to some tolerance. In this section, we experimentally investigate how this tolerance affects the quality of the approximation and the efficiency of the approach. To mimick the difficulty associated with the SQS~2 condition, we have also performed some tests where we only enforce the SQS~1 condition up to some tolerance, and not exactly. The results of our numerical tests are displayed in Figures~\ref{fig:Sensibilite_SQS1_relative} through~\ref{fig:Sensibilite_SQS2_absolute}.

Figures~\ref{fig:Sensibilite_SQS1_relative} and~\ref{fig:Sensibilite_SQS1_absolute} show the sensitivity of the variance reduction ratio upon the first order condition~\eqref{eq:cond1_fin}. Using Algorithm~\ref{algo:SQS_selection}, we generate ${\cal M} = 2000$ realizations. To investigate the robustness of our approach, we sort these realizations with respect to the error in~\eqref{eq:cond1_fin}, and successively consider 20 groups of 100 realizations that less and less accurately satisfy~\eqref{eq:cond1_fin}. On Figure~\ref{fig:Sensibilite_SQS1_relative}, the left-most circle displays the ratio $V_{\rm SQS~1}/V_{\rm MC}$ between the empirical variance $V_{\rm SQS~1}$ among the best $M=100$ realizations and the reference Monte Carlo variance $V_{\rm MC} = \Var \left[ \Big( A^\star_N \Big)_{11} \right]$. The second circle shows the ratio between the empirical variance among the next best $M=100$ realizations and the reference Monte Carlo variance $V_{\rm MC}$. We proceed similarly with all the subsequent groups of $M=100$ realizations. 

On Figure~\ref{fig:Sensibilite_SQS1_absolute}, we display the same ratio of variances in function, for each group of $M=100$ realizations, of the maximum error with which the first order condition~\eqref{eq:cond1_fin} is satisfied. Hence, the first group (left-most circle) corresponds to exactly satisfying the condition, the second group corresponds to an error between $0$ and $tol$, the third group corresponds to an error between $tol$ and $2 \, tol$, and so on and so forth.

\medskip

Figures~\ref{fig:Sensibilite_SQS2_relative} and~\ref{fig:Sensibilite_SQS2_absolute} show the sensitivity upon the second order condition~\eqref{eq:cond2_fin}. Here, we only consider realizations that satisfy~\eqref{eq:cond1_fin}. Using Algorithm~\ref{algo:SQS_selection}, we again generate ${\cal M} = 2000$ realizations and sort them according to the error in~\eqref{eq:cond2_fin}. We again successively consider 20 groups of 100 realizations that all satisfy~\eqref{eq:cond1_fin} but that less and less accurately satisfy~\eqref{eq:cond2_fin}. We present the results on Figures~\ref{fig:Sensibilite_SQS2_relative} and~\ref{fig:Sensibilite_SQS2_absolute} following the same procedure as for Figures~\ref{fig:Sensibilite_SQS1_relative} and~\ref{fig:Sensibilite_SQS1_absolute}. For instance, for the left-most circle, we plot the ratio $\dis \frac{V_{\rm SQS~2}}{V_{\rm exact~SQS~1}}$ between the variance $V_{\rm SQS~2}$ among the $M=100$ realizations that exactly satisfy the SQS~1 condition and best satisfy the SQS~2 condition on the one hand, and, on the other hand, the variance $V_{\rm exact~SQS~1}$ of the realizations that exactly satisfy the SQS~1 condition. 

\medskip

We observe that, even if the SQS conditions~\eqref{eq:cond1_fin}--\eqref{eq:cond2_fin} are not {\em exactly} satisfied, but only with some small tolerance, we obtain a significant variance reduction. We conclude that our approach is robust in this respect. 

\begin{figure}[htbp]
\centering
\psfrag{Evolution of variance REDUCTION \(reference is flat MC sample\) w.r.t. class selected. Left is best 5%, right is worst 5%}{}
\includegraphics[scale=0.6]{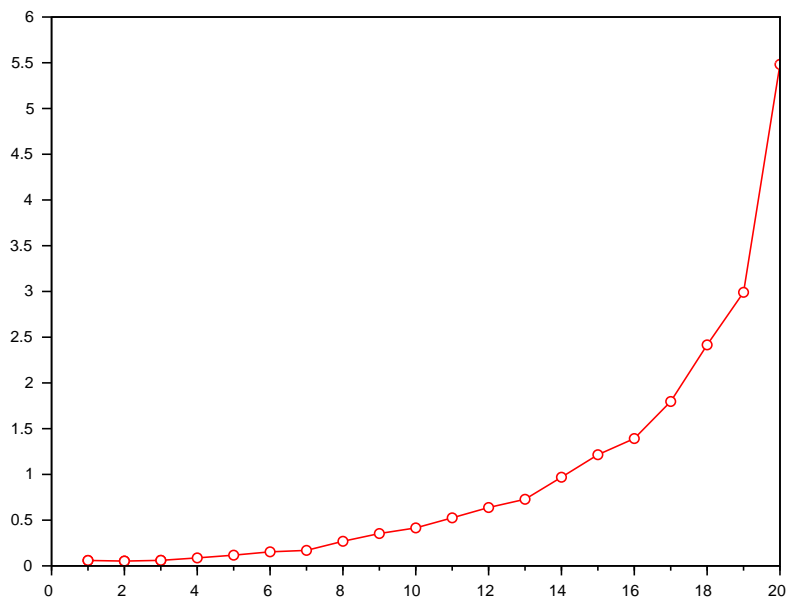}
\caption{Variance ratio $V_{\rm SQS~1}/V_{\rm MC}$ for the 20 groups of realizations (sorted according to their SQS~1 error). \label{fig:Sensibilite_SQS1_relative}}
\end{figure}

\begin{figure}[htbp]
\centering
\psfrag{Evolution of variance REDUCTION of SQS1 method \(reference is flat MC sample\) w.r.t. class selected. Bottom is absolute error.}{}
\includegraphics[scale=0.6]{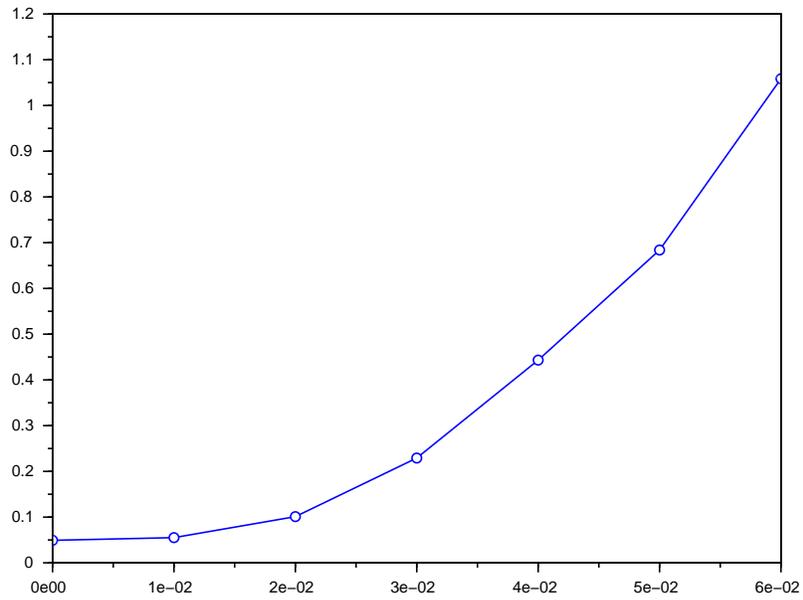}
\caption{Variance ratio $V_{\rm SQS~1}/V_{\rm MC}$ as a function of the error in~\eqref{eq:cond1_fin}. Results for only the best 7 groups (out of the 20 groups) are shown. \label{fig:Sensibilite_SQS1_absolute}}
\end{figure}

\begin{figure}[htbp]
\centering
\psfrag{Variance REDUCTION of SQS2 selection on top of perfect SQS 1}{}
\psfrag{class selected}{}
\psfrag{variance}{}
\includegraphics[scale=0.6]{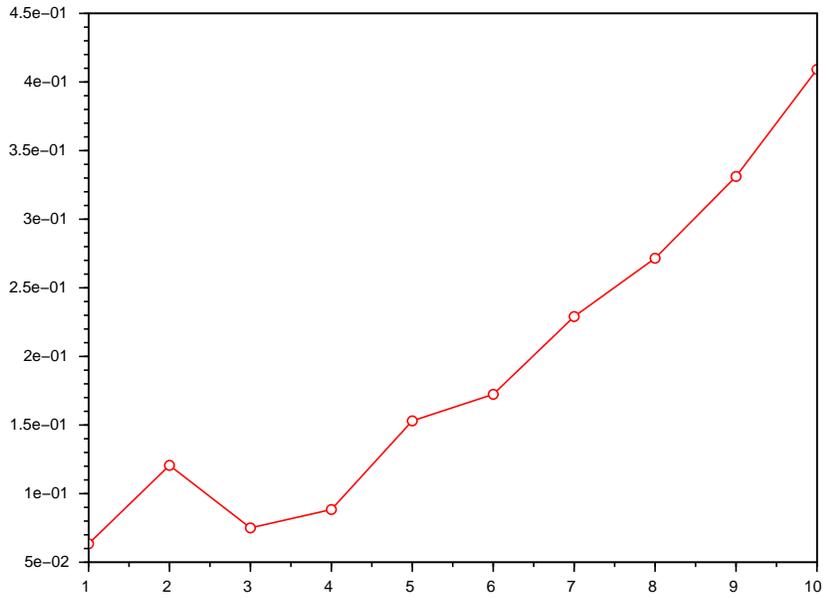}
\caption{Variance ratio $\dis \frac{V_{\rm SQS~2}}{V_{\rm exact~SQS~1}}$ for the distinct groups of realizations (sorted according to their SQS~2 error; the SQS~1 condition is exactly satisfied). Only the 10 best groups are shown. \label{fig:Sensibilite_SQS2_relative}}
\end{figure}

\begin{figure}[htbp]
\centering
\psfrag{Variance REDUCTION of SQS2 selection on top of perfect SQS 1}{}
\psfrag{value of error}{}
\psfrag{variance}{}
\includegraphics[scale=0.6]{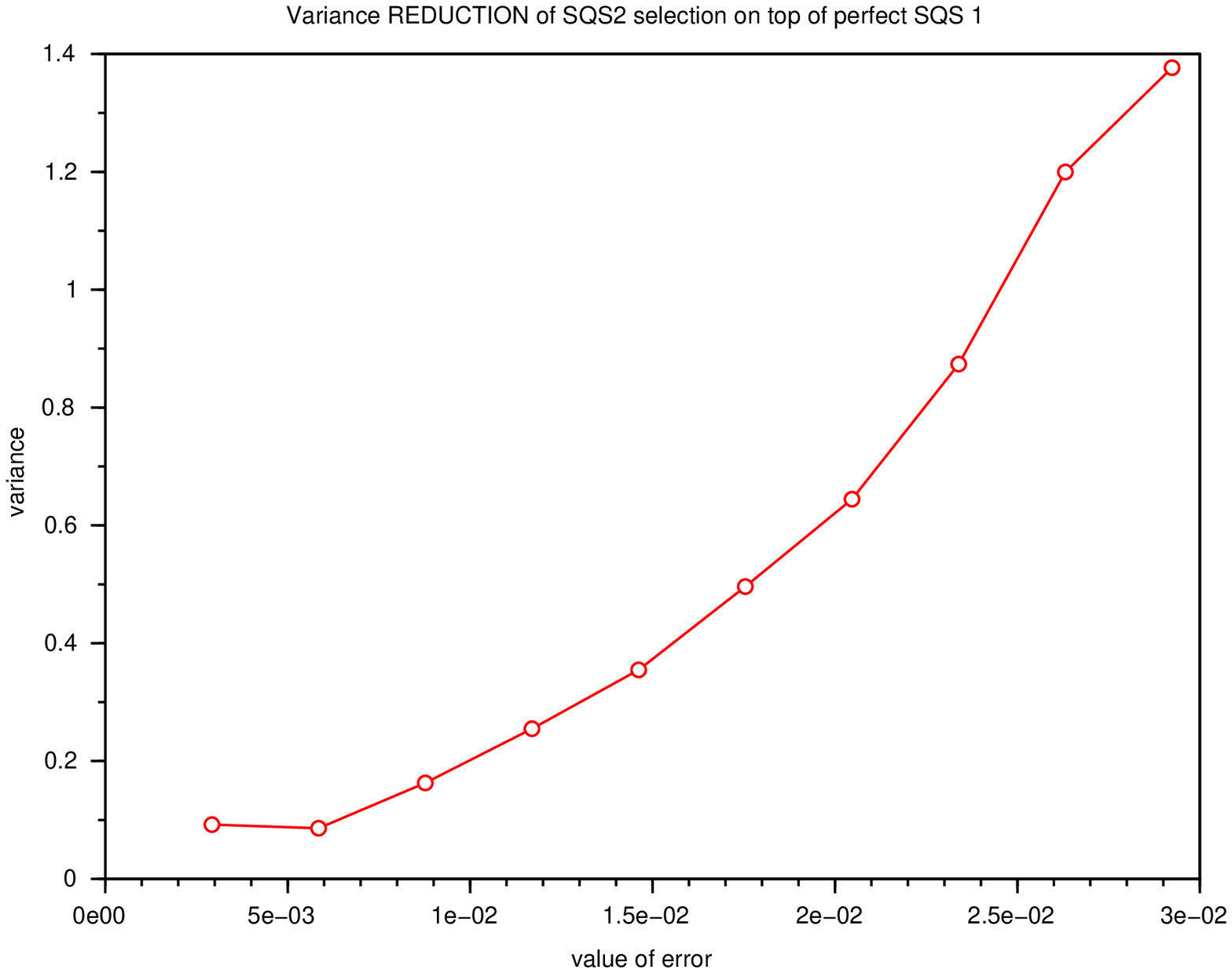}
\caption{Variance ratio $\dis \frac{V_{\rm SQS~2}}{V_{\rm exact~SQS~1}}$ as a function of the error in~\eqref{eq:cond2_fin} (the condition~\eqref{eq:cond1_fin} is exactly satisfied). Only the 10 best groups are shown. \label{fig:Sensibilite_SQS2_absolute}}
\end{figure}

\medskip

We next investigate whether enforcing the SQS~1 condition~\eqref{eq:cond1_fin} affects the probability distribution function of the left-hand side of the SQS~2 condition~\eqref{eq:cond2_fin}. Results are shown on Figure~\ref{fig:histogram}, where we plot two histograms:
\begin{itemize}
\item the distribution of the criterion SQS~2 (namely, the left-hand side of~\eqref{eq:cond2_fin}) among all realizations.
\item the conditional distribution of the criterion SQS~2 among realizations that exactly satisfy the SQS~1 condition~\eqref{eq:cond1_fin}. 
\end{itemize}
The two histograms are sufficiently close to each other to state that enforcing the SQS~1 condition does not change the distribution of the SQS~2 criterion.

\begin{figure}[htbp]
\centering
\psfrag{Distribution of SQS 2 criterion, conditioned to SQS 1 \(exact\) in red, unconditioned in black}{}
\includegraphics[scale=0.4]{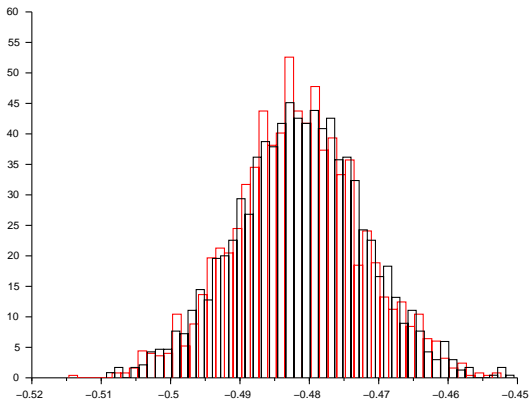}
\caption{Empirical probability distribution function of the SQS~2 criterion (black histogram: no conditioning; red histogram: the samples exactly satisfy the SQS~1 criterion). Both histograms have been computed using 100 realizations. \label{fig:histogram}}
\end{figure}

\subsection{Efficiency of the approach}
\label{sec:convergence}

In this section, we investigate how the efficiency of our approach depends (i) on the size of the truncated domain $Q_N$ and (ii) on the contrast in $A$. 

\subsubsection{Experimental error analysis}

Figure~\ref{fig:Convergence_wrt_supercell} shows the set of approximations of the first entry $\left[ A^\star \right]_{11}$ of the homogenized matrix and their respective confidence intervals. We show three curves (along with their respective confidence intervals):
\begin{itemize}
\item The standard Monte Carlo approximation, which is defined by~\eqref{eq:MC_estim_homog}. The variance is large. 
\item The approximation obtained by selecting realizations that exactly satisfy the SQS~1 condition. The variance is much smaller, leading in turn to a narrower confidence interval.
\item The approximation obtained with realizations satisfying exactly the SQS~1 condition and selected according to the SQS~2 condition (see Algorithm~\ref{algo:SQS_selection}). The variance is much smaller than when using the SQS~1 approach, even when the size of the domain $Q_N$ is small.
\end{itemize}

\begin{figure}[htbp]
\centering
\psfrag{Red = SQS 1st order, Blue = SQS 2nd order, Black- = MC, Black-- = exact}{}
\psfrag{meanvalue}{$\left[ A^\star \right]_{11}$}
\psfrag{supercell size}{\hspace{-20mm} Size $N$ of the domain $Q_N$}
\includegraphics[scale=0.6]{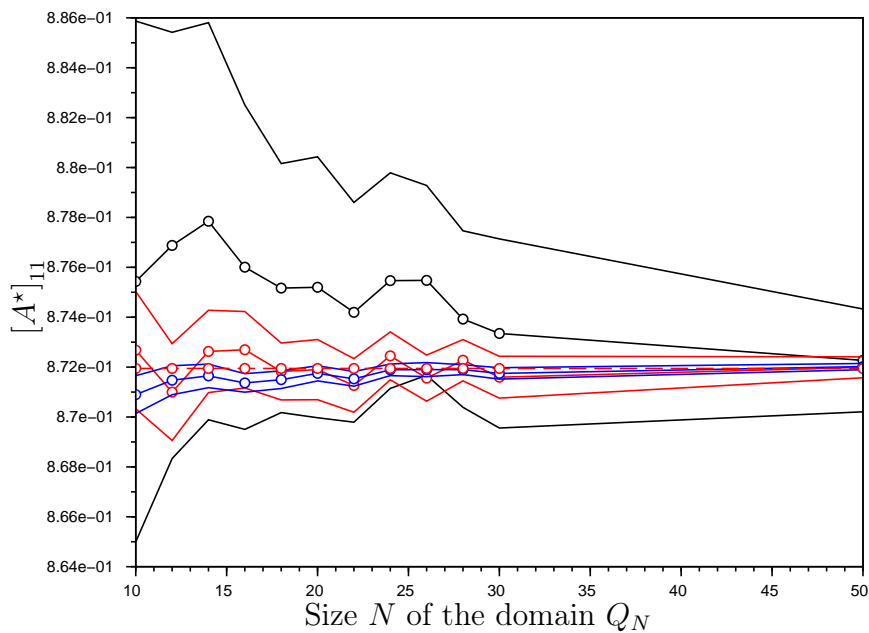}
\caption{Approximations of $\left[ A^\star \right]_{11}$ (along with confidence intervals) as a function of $N$. Black curve: Monte Carlo method. Red curve: SQS~1 method. Blue curve: SQS~2 method (see text). \label{fig:Convergence_wrt_supercell}}
\end{figure}

\medskip

Figure~\ref{fig:bias_wrt_SQS1calM_wrt_N} shows a representation of the total error as a function of the size of $Q_N$. As often in a Monte Carlo approach, computing the total error is challenging, precisely because the reference value is, by definition, in general unknown. In the specific case considered, namely the random checkerboard, the value of $A^\star$ is actually known and equal to $A^\star = \sqrt{(1+\eta)(1-\eta)} \, \text{Id}$. But for the large number of realizations and the large size of the truncated domains that we aim at considering, the total error is so small that we cannot neglect the contribution of the specific error due the finiteness of the meshsize used to solve~\eqref{eq:correcteur-random-N}. Therefore, we are bound to obtain and use as reference an approximate value $A^\star_h$ of $A^\star$ corresponding to an hypothetical finite element approximation on the whole space. As a surrogate for this $A^\star_h$, which is unknown in practice, we choose the empirical expectation of $A^\star_{N_{\rm ref}}(\omega)$ over ${\cal M}_{\rm ref} = 2000$ random realizations exactly satisfying the SQS~1 condition (with a view to use a value with the lowest possible statistical error), and for the largest domain $Q_{N_{\rm ref}}$ we can consider given the computing facilities we have access to, that is $N_{\rm ref}=50$. 
%The reference value is thus defined as
%$$
%A^\star_{\rm ref} = \EE \left[ A^\star_{N_{\rm ref}} \ | \ \text{SQS~1 condition is exactly satisfied} \right],
%$$
%which is in practice computed as an empirical mean over ${\cal M}_{\rm ref} = 2000$ realizations.

On Figure~\ref{fig:bias_wrt_SQS1calM_wrt_N}, we display three curves:
\begin{itemize}
\item The total error of the standard Monte Carlo method, defined as
$$
\text{total error} 
=
\left| \frac{1}{M} \sum_{m=1}^M A^\star_N(\omega_m) - A^\star_{\rm ref} \right|.
$$
\item The other two curves show the same quantity, where the $M$ environments considered now satisfy either the SQS~1 condition or that condition together with the SQS~2 condition.
\end{itemize} 
The results obtained using the SQS~2 approach are in general comparable to, and often better than, those obtained with the SQS~1 approach. More accurate estimates of the reference value $A^\star_h$ would probably help in clarifying the superiority of SQS~2 over SQS~1 in terms of total accuracy. As will now be seen, the superiority of SQS~2 in terms of variance (which, in some sense, is the key point for practice) is definite.

\begin{figure}[htbp]
\centering
\psfrag{Decay rate base = -1.4323923, SQS = -1.3980905, SQS 2nd order = -1.8539045, SQS is in red, SQS2 in blue}{}
\psfrag{log of bias}{}
\psfrag{log of supercell}{}
\includegraphics[scale=0.6]{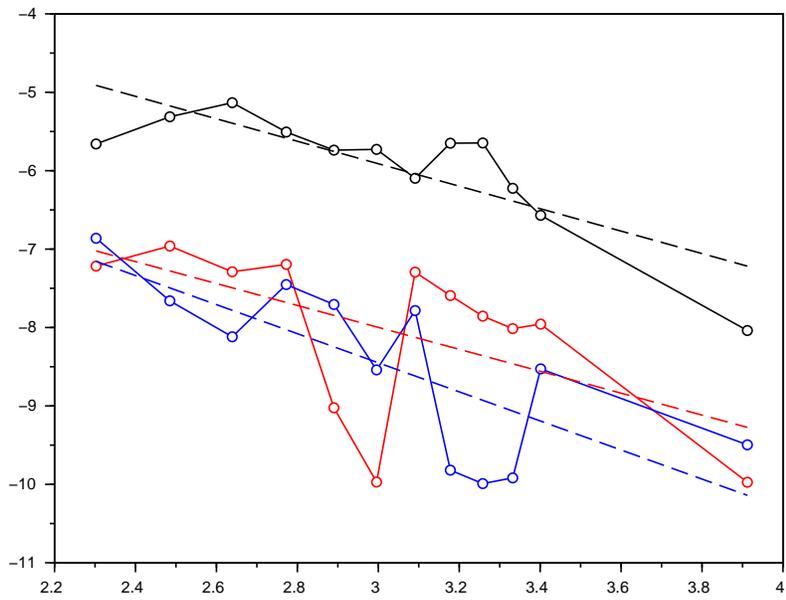}
\caption{log-log plot of the total error as a function of $N$ (natural logarithm). Black curve: Monte Carlo method. Red curve: SQS~1 method. Blue curve: SQS~2 method (see text). \label{fig:bias_wrt_SQS1calM_wrt_N}}
\end{figure}

\medskip

Figure~\ref{fig:var_wrt_N} shows the empirical variance of the different approximations of $\left[ A^\star \right]_{11}$ as a function of the size of $Q_N$. We again display three curves:
\begin{itemize}
\item The standard Monte Carlo approximation defined by~\eqref{eq:MC_estim_homog}. 
\item The approximation obtained by selecting realizations that exactly satisfy the SQS~1 condition. 
\item The approximation obtained with realizations exactly satisfying the SQS~1 condition and selected according to the SQS~2 condition (see Algorithm~\ref{algo:SQS_selection}).
\end{itemize} 
We observe that, each time we consider an additional SQS condition, the empirical variance of the approximation is significantly reduced (even if this SQS condition is not exactly enforced; recall that we consider here only the 5 \% best samples in terms of the SQS~2 condition, but that we are unable to enforce it exactly). On our test-case, enforcing the SQS~1 condition leads to a variance 20 times smaller than that of the standard Monte Carlo approach, while additionally enforcing the SQS~2 condition leads to an additional variance reduction of a factor of 10.

We also observe on Figure~\ref{fig:var_wrt_N} that all variances decay as $\lambda/|Q_N|$, where 
$$
\lambda_{\rm SQS~2} < \lambda_{\rm exact~SQS~1} < \lambda_{\rm MC}.
$$
This corroborates in higher dimension the behaviour predicted in Section~\ref{sec:theorique_degrade}. In particular, the gain in variance does not decrease when the size of $Q_N$ becomes larger.

\begin{figure}[htbp!]
\centering
\psfrag{Decay rate = -2.008467, -2.056233, -2.1745635, SQS 2nd order is in blue}{}
\psfrag{log of variance}{}
\psfrag{log of supercell}{}
\includegraphics[scale=0.6]{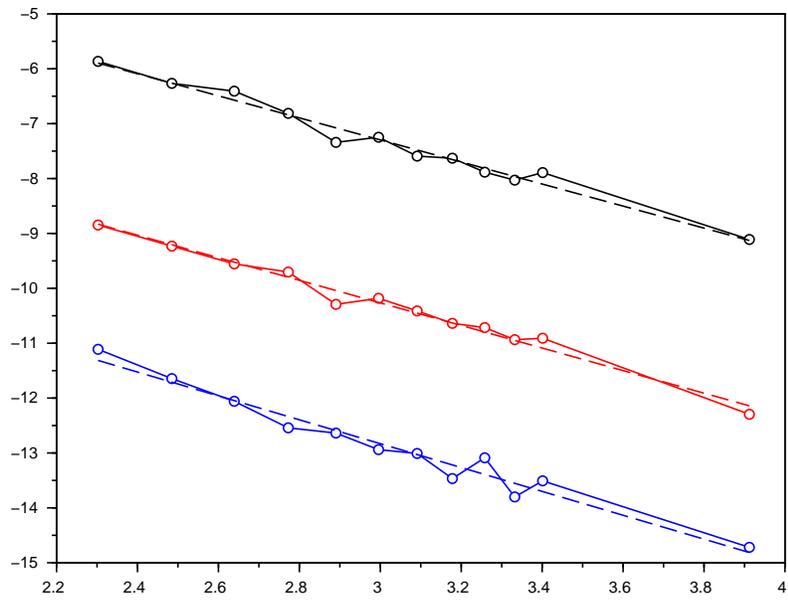}
\caption{log-log plot of the variance as a function of $N$ (natural logarithm). Black curve: Monte Carlo method. Red curve: SQS~1 method. Blue curve: SQS~2 method. \label{fig:var_wrt_N}}
\end{figure}

\subsubsection{Sensitivity to the contrast}

We eventually investigate how the contrast in the field $A$ affects the gain in variance. Results are shown in Table~\ref{tab:contrast}. We observe that the gain decreases when the contrast increases. Note that this is also the case with the antithetic variable and the control variate techniques that we have previously studied (see~\cite{BlancCostaouecLeBrisLegoll2012a,LegollMinvielle2015a,LegollMinvielle2015b}).

However, our SQS~2 approach still yields a significant gain of a factor of 10 when the contrast is equal to 20.

\begin{table}[htbp]
\begin{center}
\begin{tabular}{c|c|c|c|c|c}
\text{Contrast} & $V_{\rm MC}$ & $V_{\rm exact~SQS~1}$ & $V_{\rm SQS~2}$ & $\dis \frac{V_{\rm MC}}{V_{\rm exact~SQS~1}}$ & $\dis \frac{V_{\rm MC}}{V_{\rm SQS~2}}$ \\
\hline
1.22 & 0.0000273 &   5.801e-08 &   6.858e-10 &   470  &   39821 \\  
1.50 & 0.0001097 &   0.0000009 &   1.585e-08 &   118  &   6921 \\  
1.86 & 0.0002488 &   0.0000047 &   0.0000001 &   52.6 &   1996 \\   
2.33 & 0.0004478 &   0.0000151 &   0.0000006 &   29.5 &   720 \\   
3.00 & 0.0007118 &   0.0000379 &   0.0000024 &   18.8 &   296 \\  
4.00 & 0.0010496 &   0.0000814 &   0.0000080 &   12.8 &   131 \\  
5.67 & 0.0014769 &   0.0001600 &   0.0000244 &   9.23 &   60.5 \\  
9.00 & 0.0020289 &   0.0003021 &   0.0000739 &   6.71 &   27.4 \\  
19.0 & 0.0028330 &   0.0006061 &   0.0002554 &   4.67 &   11.1 \\  
\end{tabular}
\caption{For various values of the contrast, we show the Monte Carlo variance (column \#2), the variance of the SQS~1 method (column \#3) and the variance of the SQS~2 method (column \#4). We next show the variance ratio $V_{\rm MC}/V_{\rm exact~SQS~1}$ for the SQS~1 approach (column \#5) and the variance ratio $V_{\rm MC}/V_{\rm SQS~2}$ for the SQS~2 approach (column \#6). The size of $Q_N$ is fixed at $N=20$. \label{tab:contrast}}
\end{center}
\end{table}
% a nouveau, h=0.2
% cf emails william à moi lami, 2juillet15

\subsection{A more general geometry}
\label{sec:mech}

In order to show that the approach carries over to other settings involving more general geometries than the setting considered above, we briefly consider in the present final section  a linear elasticity problem, for a two-phase composite material with random inclusions. The radii $r_j(\omega)$ of the inclusions are i.i.d. random variables satisfying
$$
r_j(\omega) = \sqrt{(M-m) \sqrt{U_j(\omega)} + m},
$$
where $U_j(\omega)$ are i.i.d. variables uniformly distributed in $[0,1]$. The parameters $M$ and $m$ are such that the minimum (resp. maximum) inclusion radius is 0.125 (resp. 0.45). The inclusions centers are distributed according to a Poisson point process, and we additionally impose that inclusions do not overlap (see Figure~\ref{fig:inclusions}). We consider the truncated domain $Q_N = [0,5]^2$. Each microstructure contains 25 inclusions. If some part of an inclusion falls outside of $Q_N$, it is reproduced on the other side of $Q_N$ by periodicity.

\begin{figure}[htbp]
\centering
\includegraphics[scale=0.25,angle=-90]{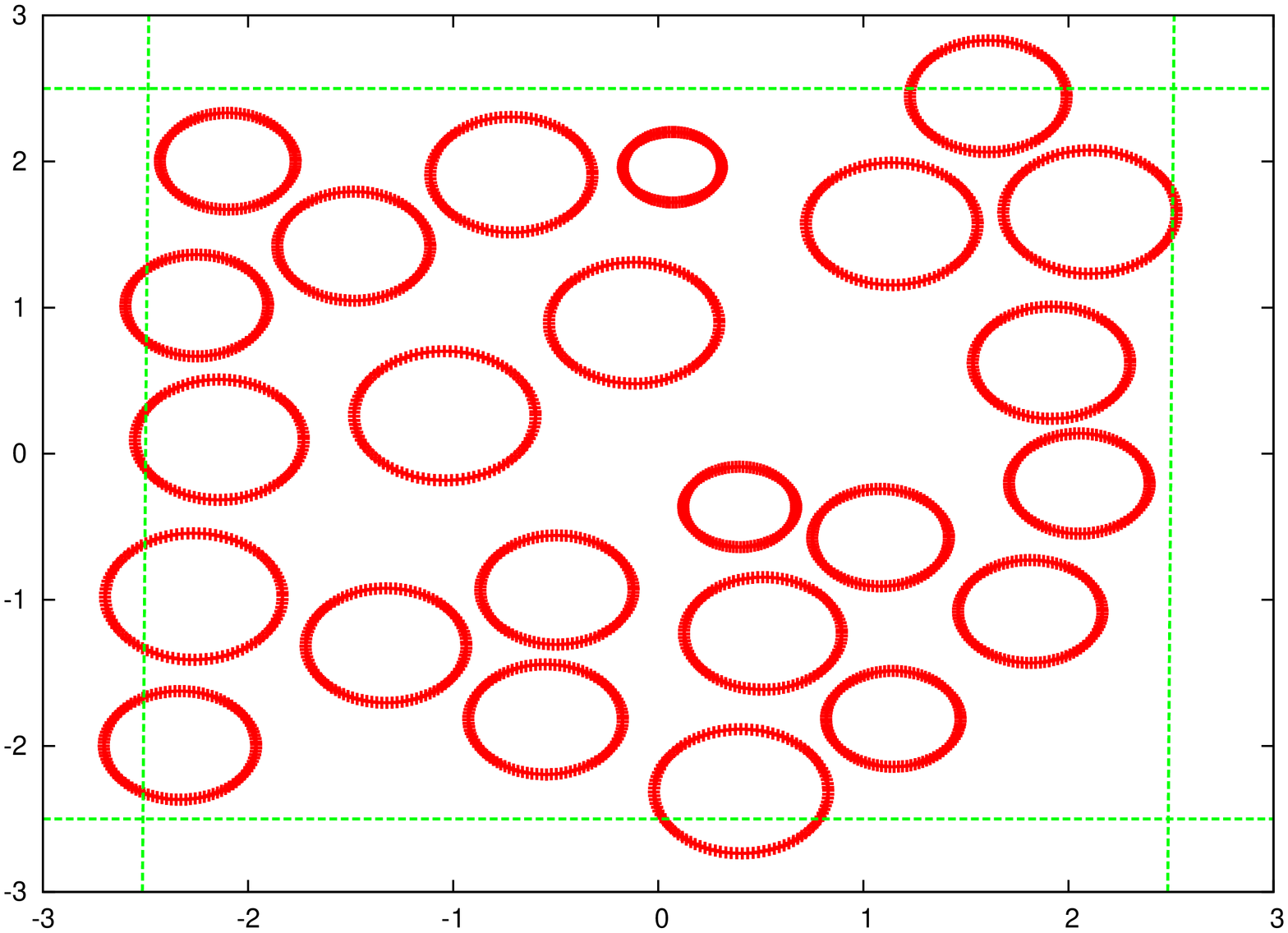}
\includegraphics[scale=0.25,angle=-90]{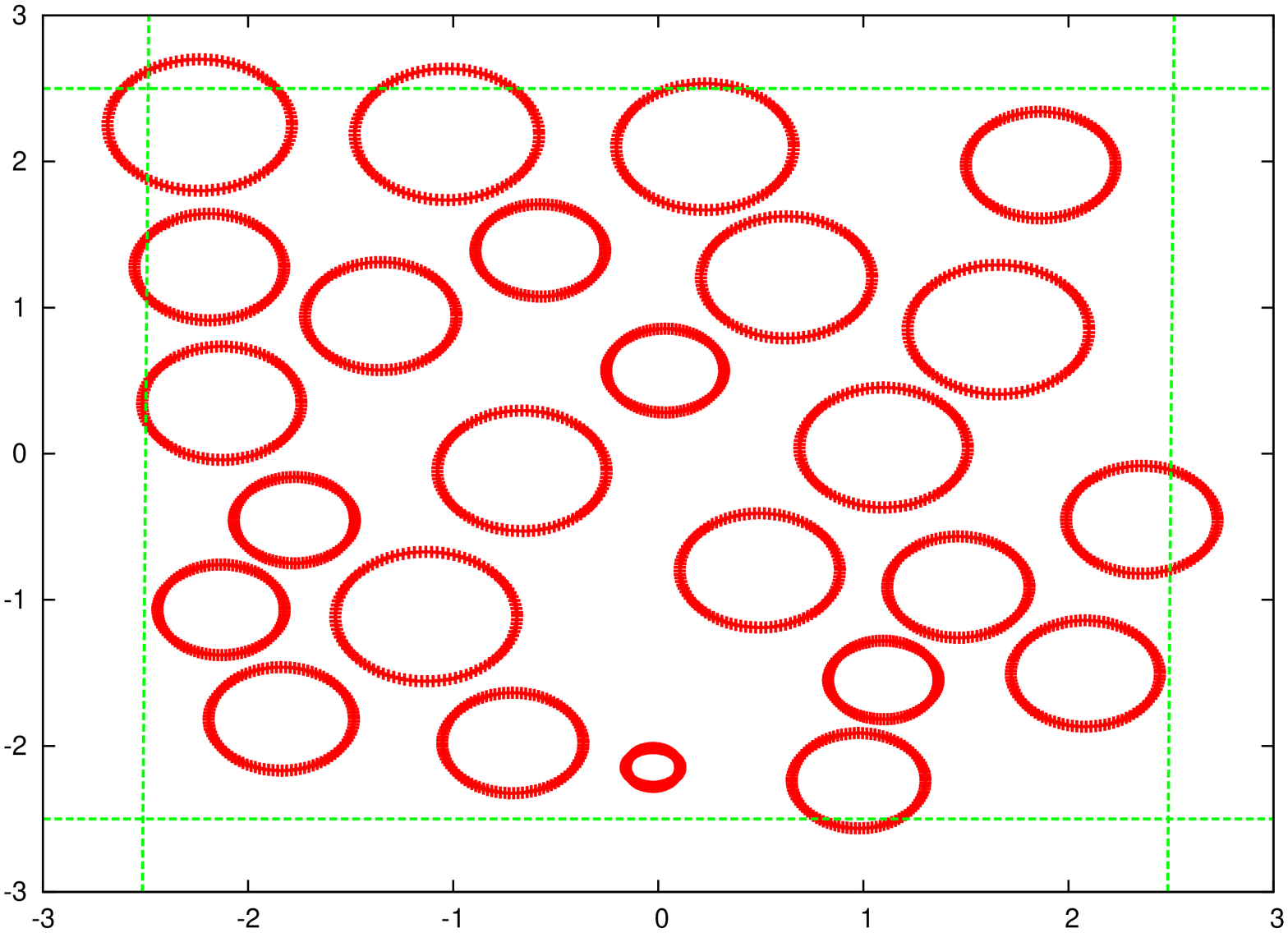}
\caption{Two realizations of the microstructure geometry. \label{fig:inclusions}}
\end{figure}

The inclusions (resp. the background) are modeled by a isotropic linear elasticity tensor, with a uniform Poisson ratio $\nu = 0.3$ and a Young modulus $E=7$ (resp. $E=1$).

\medskip

For that problem, in the spirit of the SQS~1 approach described above, we select the microstructures such that the volumic fraction $\theta_N(\omega)$ of inclusions within the domain $Q_N$, for each random realization considered, agrees as accurately as possible with its asymtotic value $\dis \theta^\star = \lim_{N \to \infty} \theta_N(\omega)$.

The results are shown in Table~\ref{tab:mech}. We examine two entries of the homogenized elasticity tensor, and two methods: 
\begin{itemize}
\item The classical Monte Carlo approach, for which we generate $M=100$ i.i.d. microstructures.
\item The SQS approach, in which we generate ${\cal M} = 2000$ i.i.d. microstructures, and next consider the $M=100$ microstructures for which $\left| \theta_N(\omega) - \theta^\star \right|$ is the smallest. 
\end{itemize}
In both cases, we solve $M$ correctors problems and we compute the empirical expectation and variance of the homogenized elasticity tensor. We observe that our approach provides a variance reduction of a factor close to 20, while the bias is essentially constant.

\begin{table}[htbp]
\begin{center}
\begin{tabular}{| c | | c | c | | c | c |}
\hline
& $\left[ A^\star_N(\omega) \right]_{1111}$ & $\left[ A^\star_N(\omega) \right]_{1111}$ & $\left[ A^\star_N(\omega) \right]_{1122}$ & $\left[ A^\star_N(\omega) \right]_{1122}$
\\
\hline
& Exp. & Var. & Exp. & Var.
\\
\hline
MC approach & 2.522 & 0.0136 & 1.016 & 0.00184 
\\
\hline
SQS approach & 2.519 & 0.000532 & 1.014 & 0.000101
\\
\hline
\end{tabular}
\caption{Empirical expectation (Exp.) and variance (Var.) for two entries of the homogenized elasticity tensor, computed with the Monte Carlo approach (MC) or our approach (SQS). \label{tab:mech}}
\end{center}
\end{table}
% ceci vient de /home/legoll/michel_bornert/reduc_variance_ete-2013_ete-2014/generation_inclusions/test_plus_long/resu_cluster/exploitation

\section*{Acknowledgements}

The first two authors would like to thank E.~Canc\`es for introducing them to the SQS approach in the context of solid state physics, pointing out to them References~\cite{vonPezoldDickFriakNeugebauer2010,WeiFerreiraBernardZunger1990,ZungerWeiFerreiraBernard1990}, as well as for several stimulating discussions in the early stages of this work. The authors also thank J.-D.~Deuschel for his helpful comments and for pointing out Reference~\cite{dembo_zeitouni_1996}, and X.~Blanc for his remarks on a draft version of this article.

This work was partially supported by ONR under Grant N00014-12-1-0383 and by EOARD under Grant FA8655-13-1-3061. This work has benefited from a French government grant managed by ANR within the frame of the national program Investments for the Future ANR-11-LABX-022-01.

\appendix

\section{Proof of Lemma~\ref{lem:decompo}}
\label{app:proof}

We follow the arguments of the proof of~\cite[Lemma 3.2]{BlancCostaouecLeBrisLegoll2012b}. 

The existence and uniqueness (up to the addition of a constant) of $\phi_1$ solution to~\eqref{eq:def_u1_k} is established in~\cite[Lemma 3.1]{BlancCostaouecLeBrisLegoll2012b}. We next point out that~\eqref{eq:def_u1_bar} admits a unique (up to the addition of a constant) solution in $H^1_{\rm per}(Q)$. It is a simple consequence of the Lax-Milgram lemma. 

We now prove that the sum in~\eqref{eq:def_u1} is a convergent series in $L^2(Q \times \Omega)$. For this purpose, we compute the norm of the remainder of the series, using the notation $\overline{X}_k(\omega) = X_k(\omega) - \EE[X_k]$:
\begin{eqnarray*}
&& \left\| \sum_{|k|\geq N+1} \overline{X}_k \nabla \phi_1(\cdot -k) \right\|^2_{L^2(Q \times \Omega)} 
\\
&=& 
\sum_{|k|\geq N+1} \sum_{|\ell|\geq N+1} \EE \left[ \overline{X}_k \overline{X}_\ell \right] \int_Q \nabla \phi_1(y-k) \cdot \nabla \phi(y-\ell) \, dy 
\\
&\leq& 
\sum_{|k|\geq N+1} \sum_{|\ell|\geq N+1} \left| \Cov (X_k,X_\ell) \right| \ \| \nabla \phi_1 \|_{L^2(Q-k)} \ \| \nabla \phi_1 \|_{L^2(Q-\ell)}
\\
&\leq& 
\sum_{|k|\geq N+1} \sum_{|\ell|\geq N+1} \left| \Cov (X_k,X_\ell) \right| \ \| \nabla \phi_1 \|^2_{L^2(Q-k)},
\end{eqnarray*}
where we have used at the last line the discrete Cauchy-Schwarz inequality between the sequences $\left| \Cov (X_k,X_\ell) \right|^{1/2} \ \| \nabla \phi_1 \|_{L^2(Q-k)}$ and $\left| \Cov (X_k,X_\ell) \right|^{1/2} \ \| \nabla \phi_1 \|_{L^2(Q-\ell)}$. We next write, using the stationarity of $X_k$ and~\eqref{eq:hyp_cov}, that
\begin{eqnarray*}
&& \left\| \sum_{|k|\geq N+1} \overline{X}_k \nabla \phi_1(\cdot -k) \right\|^2_{L^2(Q \times \Omega)} 
\\
&\leq& 
\sum_{|k|\geq N+1} \| \nabla \phi_1 \|^2_{L^2(Q-k)} \sum_{|\ell|\geq N+1} \left| \Cov (X_k,X_\ell) \right|
\\
&\leq& 
\mathcal{C} \sum_{|k|\geq N+1} \| \nabla \phi_1 \|^2_{L^2(Q-k)}. 
\end{eqnarray*}
The above right-hand side converges to $0$ as $N \to \infty$ since $\nabla \phi_1 \in \left( L^2(\RR^d) \right)^d$. 

Hence, the right-hand side of~\eqref{eq:def_u1} defines a function $T \in \left( L^2(Q \times \Omega) \right)^d$. As $\partial_i T_j = \partial_j T_i$, there exists a function $\widetilde{u}_1$ such that
$$
\nabla \widetilde{u}_1 = T = 
\EE[X_0] \nabla \overline{u_1} + \sum_{k\in \ZZ^d} \Big( X_k(\omega) - \EE[X_k] \Big) \nabla \phi_1(\cdot-k).
$$
As $\overline{u_1}$ is $\ZZ^d$-periodic, we infer from the above equality that
\begin{equation}
\label{eq:util2}
\nabla \widetilde{u}_1 \text{ is stationary and }
\int_Q {\mathbb E}(\nabla \widetilde{u}_1) = 0.
\end{equation}
Next, we compute
$$
C_0 \nabla \widetilde{u}_1 = \EE[X_0] C_0 \nabla \overline{u_1} + \sum_{k\in \ZZ^d} \Big( X_k(\omega) - \EE[X_k] \Big) C_0 \nabla \phi_1(\cdot-k).
$$
Taking the divergence of this equation and using~\eqref{eq:C_0_0} and~\eqref{eq:C_1_per}, we thus find that, in the distribution sense, 
\begin{eqnarray}
\nonumber
&&-\dive\left[C_0 \nabla \widetilde{u}_1 \right]
\\
\nonumber
&=&  
\sum_{k\in \ZZ^d} -\Big( X_k(\omega) - \EE[X_k] \Big) \dive \left[C_0 \nabla \phi_1(\cdot-k) \right]
%\\ 
%\nonumber
%& &
- \EE[X_0] \dive\left[C_0 \nabla \overline{u_1} \right] 
\\
\nonumber
&=& 
\sum_{k\in \ZZ^d} \Big( X_k(\omega) - \EE[X_k] \Big) \dive \left[\1_{Q+k} C_1 p \right]
%\\
%\nonumber
%& & 
+ \EE[X_0] \dive \left[ C_1 p \right]
\\
\label{eq:util3}
&=& 
\dive\left[ \chi(\cdot,\omega) C_1 p \right].
\end{eqnarray}
Collecting~\eqref{eq:util2} and~\eqref{eq:util3}, we see that $\widetilde{u}_1$ solves~\eqref{eq:correcteur_cascade1}. As the solution to this equation is unique up to the addition of a (possibly random) constant $C(\omega)$, we obtain that $\widetilde{u}_1 = u_1 + C(\omega)$, hence proving~\eqref{eq:def_u1}. 

\bibliography{these}

\end{document}